\documentclass{amsart}
\usepackage{amssymb,amsmath,latexsym,times,tikz,hyperref,mathrsfs}

\hypersetup{colorlinks=true, linkcolor=blue, citecolor=magenta}

\numberwithin{equation}{section}

\theoremstyle{plain}
\newtheorem{thm}{Theorem}[section]

\newtheorem{cor}[thm]{Corollary}

\newtheorem{lemma}[thm]{Lemma}

\theoremstyle{definition}
\newtheorem{example}{Example}

\newcommand{\del}{\backslash}
\DeclareMathOperator{\cl}{cl}
\DeclareMathOperator{\cy}{cy}

\title[The natural matroid of a polymatroid]{The natural matroid of an
  integer polymatroid} \author[J.~Bonin]{Joseph E.~Bonin}
\address[J.~Bonin]
{Department of Mathematics\\ The George Washington University\\
  Washington, D.C. 20052, USA} \email{jbonin@gwu.edu}
\author[C.~Chun]{Carolyn Chun}
\address[C.~Chun]{United States Naval Academy\\
  Mathematics Department\\
  Annapolis, MD, 21402, USA} \email{chun@usna.edu}
\author[T.~Fife]{Tara Fife} \address[T.~Fife]{School of Mathematical
  Sciences \\ Queen Mary University of London\\
  Mile End Road, London E1 4NS, United Kingdom}
\email{fi.tara@gmail.com} \date{\today}

\begin{document}

\begin{abstract}
  The natural matroid of an integer polymatroid was introduced to show
  that a simple construction of integer polymatroids from matroids
  yields all integer polymatroids.  As we illustrate, the natural
  matroid can shed much more light on integer polymatroids.  We focus
  on characterizations of integer polymatroids using their bases,
  their circuits, and their cyclic flats along with the rank of each
  cyclic flat and each element; we offer some new characterizations
  and insights into known characterizations.
\end{abstract}

\maketitle

\section{Introduction}\label{sec:intro}

A \emph{polymatroid} is a pair $P=(E,\rho)$ where $E$ is a finite set
and the real-valued function $\rho: 2^E\to \mathbb{R}$, the \emph{rank
  function} of $P$, has the following properties:
\begin{enumerate}
\item $\rho$ is \emph{normalized}, that is, $\rho(\emptyset) = 0$,
\item $\rho$ is \emph{non-decreasing}, that is, if
  $A\subseteq B\subseteq E$, then $\rho(A)\leq \rho(B)$, and
\item $\rho$ is \emph{submodular}, that is,
  $\rho(A\cup B)+\rho(A\cap B)\leq \rho(A) + \rho(B)$ for all
  $A, B\subseteq E$.
\end{enumerate}
Less formally, we often talk about a polymatroid $\rho$ on $E$.  A
\emph{$k$-polymatroid}, where $k\in\mathbb{R}$ and $k>0$, is a
polymatroid $(E,\rho)$ for which $\rho(e)\leq k$ for all $e \in E$.
For much of this paper, we are concerned with \emph{integer
  polymatroids} (also called \emph{discrete polymatroids}), that is,
polymatroids $\rho$ where the rank $\rho(A)$ of each set $A$ is in the
set $\mathbb{N}$ of nonnegative integers.

Intuitively, a matroid (an integer $1$-polymatroid) can be thought of
as a configuration of points, lines, planes, and so on, in which each
of the elements that make up these objects has rank $0$ (loops) or $1$
(points).  An integer polymatroid is the natural generalization in
which the elements are not limited to points and loops; we also allow,
as the elements, lines (elements of rank $2$), planes (elements of
rank $3$), and so on.  Not surprisingly, every integer polymatroid
comes from a matroid, as the following result of
\cite{helgason,ll,mcd} states.

\begin{thm}\label{thm:reppoly}
  A function $\rho:2^E \to \mathbb{N}$ is an integer polymatroid if
  and only if there is a matroid $M$ on a set $E'$ and a function
  $\phi:E \to 2^{E'}$ with
  $\rho(A) = r_M\bigl(\bigcup_{e\in A}\phi(e)\bigr)$ for all
  $A \subseteq E$.
\end{thm}

Helgason \cite{helgason} introduced the natural matroid to prove this
result.  Geometrically, we get the natural matroid by, for each
element $e$ of $E$, replacing $e$ by a set $\phi(e)$ of $\rho(e)$
points that are placed freely in $e$; thus, a line is replaced by two
points that are put freely on the line, and a plane by three points
that are placed freely on the plane, and so on.  (Section
\ref{sec:nat} has a precise definition of the natural matroid.)  Many
important properties of integer polymatroids are closely linked to
properties of its natural matroid.  For instance, Oxley, Semple, and
Whittle \cite{splitter} showed that an integer $2$-polymatroid is
$3$-connected if and only if it has no loops and its natural matroid
is $3$-connected.  We study the natural matroid in its own right.

In Section \ref{sec:nat}, we review the definition of the natural
matroid and prove two results that make it easy to verify that a
matroid $M$ is the natural matroid of an integer polymatroid $\rho$.
We show that, for an integer polymatroid $\rho$ on $E$ that is the sum
of the rank functions of matroids $M_1,M_2,\ldots,M_k$ on $E$, the
natural matroid of $\rho$ is the matroid union of certain extensions
of $M_1,M_2,\ldots,M_k$ by loops and elements parallel to those in
these matroids.

Herzog and Hibi \cite{HerzogHibi} treat characterizations of integer
polymatroids using bases and exchange properties.  In Section
\ref{sec:bases}, we show how these results follow easily by observing
that the bases of an integer polymatroid are the type vectors of the
bases of its natural matroid.

Viewing the bases of an integer polymatroid as the type vectors of the
bases of its natural matroid suggests developing an analogous theory
for circuits.  We do this in Section \ref{sec:circuits}, where, in
Theorem \ref{thm:circchar}, we introduce circuit axioms for integer
polymatroids.

Cyclic flats of matroids, along with their ranks, provide relatively
compact descriptions of matroids that allow one to focus on crucial
features when, for instance, defining certain matroid constructions
(e.g., see \cite{intertwine,diffconfig,cyc,jens}); this perspective is
also useful in applications such as coding theory (e.g., see
\cite{coding}).  In Section \ref{sec:cyclic}, we show that some
results about cyclic flats lift from the natural matroid of an integer
polymatroid to the polymatroid.  In the case of integer polymatroids,
this gives another perspective on recent work by Csirmaz
\cite{cyclicpoly} characterizing all polymatroids via their cyclic
flats and the ranks of these flats and of singleton sets.  A key
result behind this characterization is the formula that gives the rank
function of a polymatroid $\rho$ on $E$ using only its values on
cyclic flats and singleton sets, namely,
$$\rho(A) = \min\{\rho(X) + \sum_{i\in A-X}\rho(i) \,:\, X \in
\mathcal{Z}_\rho\},$$ where $\mathcal{Z}_\rho$ is the lattice of
cyclic flats of $\rho$.  For a subset $A$ of $E$, we consider the set
$\mathcal{R}_\rho(A)$ of cyclic flats $X$ that yield this minimum.  We
show that $\mathcal{R}_\rho(A)$ is a sublattice of $\mathcal{Z}_\rho$,
we identify its least and greatest elements, and we show that each
pair of flats in $\mathcal{R}_\rho(A)$ is a modular pair.

Our matroid notation follows Oxley \cite{oxley}.  For a positive
integer $n$, let $[n]$ be the set $\{1,2,\ldots,n\}$.  We often take
the ground set of an integer polymatroid $\rho$ to be $[n]$ since this
provides a natural correspondence between the elements of $\rho$ and
the entries in $n$-tuples.  For $n\in\mathbb{N}$, the set of
nonnegative integers, let $[n]_0$ be the set $\{0,1,2,\ldots,n\}$.

For a polymatroid $\rho$ on $E$ and for $A\subseteq E$, the
\emph{deletion $\rho_{\del A}$} and \emph{contraction $\rho_{/A}$},
both on $E-A$, are defined by $\rho_{\del A}(X) = \rho(X)$ and
$\rho_{/A}(X) = \rho(X\cup A)-\rho(A)$ for all $X\subseteq E-A$. The
\emph{minors} of $\rho$ are the polymatroids of the form
$(\rho_{\del A})_{/B}$ (equivalently, $(\rho_{/B})_{\del A}$) for
disjoint subsets $A$ and $B$ of $E$. The \emph{$k$-dual} $\rho^*$ of a
$k$-polymatroid $\rho$ on $E$ is the $k$-polymatroid that is given by
$\rho^*(X)=k|X|-\rho(E)+\rho(E-X)$ for all $X\subseteq E$.  The
\emph{direct sum} $\rho_1\oplus\rho_2$ of polymatroids $\rho_1$ and
$\rho_2$ on disjoint sets $E_1$ and $E_2$ is defined by
$(\rho_1\oplus\rho_2)(X) = \rho_1(X\cap E_1)+\rho_2(X\cap E_2)$ for
$X\subseteq E_1\cup E_2$.  A polymatroid that is not a direct sum of
two polymatroids on nonempty sets is \emph{connected}.

\section{The natural matroid of an integer polymatroid}\label{sec:nat}

The construction of the natural matroid uses Theorem
\ref{thm:firstgen} below, due to McDiarmid \cite{mcd}, which
strengthens an earlier result of Edmonds and Rota.  (Theorem
\ref{thm:firstgen} is treated in \cite{oxley,Welsh}.)  Consider a
collection $L$ of subsets of a set $E$ that includes $E$ and
$\emptyset$, and that is closed under intersection; thus, under
inclusion, $L$ is a lattice, and for $A,B\in L$, their meet
$A\wedge B$ is $A\cap B$, but their join $A\vee B$ need not be
$A\cup B$.  For such a lattice $L$, a function
$\sigma : L \to \mathbb{N}$ is \emph{submodular} if
$\sigma(A \vee B) + \sigma(A \cap B)\leq \sigma(A) + \sigma(B)$ for
all $A, B \in L$.

\begin{thm}\label{thm:firstgen}
  Let $L$ be a lattice of subsets of $E$ that contains $\emptyset$ and
  $E$, and is closed under intersection.  Let
  $\sigma : L \to \mathbb{N}$ be submodular with
  $\sigma(\emptyset) = 0$. Define $r : 2^E \to \mathbb{N}$ by
  \begin{equation}\label{eq:subtorank}
    r(Y) = \min\{\sigma(S) + |Y - S| \,:\, S \in L\},
  \end{equation}
  for $Y \subseteq E$.  The function $r$ is the rank function of a
  matroid on $E$; its independent sets are the subsets $I$ of $E$ for
  which $|I \cap S| \leq \sigma(S)$ for all $S \in L$. 
\end{thm}

Given an integer polymatroid $\rho$ on a set $E$, its natural matroid
$M_\rho$ is defined as follows.  For each $i\in E$, let $X_i$ be a set
of $\rho(i)$ elements so that the sets $X_i$, for all $i\in E$, are
pairwise disjoint.  For $A\subseteq E$, set
$$X_A = \bigcup_{i\in A}X_i$$ and let $E'=X_E$.  
Let $L=\{X_A\,:\,A\subseteq E\}$.  Now $L$ is a lattice of subsets of
$E'$ with $\emptyset, E'\in L$,
$X_A\vee X_B=X_A\cup X_B = X_{A\cup B}$, and
$X_A\wedge X_B=X_A\cap X_B = X_{A\cap B}$.  Define
$\sigma : L \to \mathbb{N}$ by $\sigma(X_A) = \rho(A)$.  Since $\rho$
is submodular, so is $\sigma$.  The \emph{natural matroid} of $\rho$,
denote $M_\rho$, is the matroid on $E'$ whose rank function is given
by Equation (\ref{eq:subtorank}).  The choice of the sets $X_i$ is not
unique, but the natural matroid is well-defined up to relabeling the
elements in $E'$.

\begin{cor}\label{cor:indepinnatmtd}
  A subset $I$ of $E'$ is independent in $M_\rho$ if and only if
  $|I\cap X_A|\leq \rho(A)$ for every $A\subseteq E$.
\end{cor}

Since $\rho$ is submodular and non-decreasing, if $A,S\subseteq E$,
then
$$\rho(A)\leq \rho(A\cap S)+\sum_{i\in A-S}\rho(i)
\leq \rho(S)+\sum_{i\in A-S}\rho(i).$$ It follows that
$r_{M_\rho}(X_A)=\rho(A)$ for all $A\subseteq E$.  Theorem
\ref{thm:reppoly} follows by letting $M$ be $M_\rho$ and defining
$\phi:E\to 2^{E'}$ by $\phi(i)=X_i$.

The next lemma simplifies proving that a matroid is the natural
matroid of $\rho$.  Recall that two elements $a$ and $b$ of a matroid
$M$ on $E$ are \emph{clones} if the permutation of $E$ given by the
$2$-cycle $(a,b)$ (i.e., switching $a$ and $b$) is an automorphism of
$M$.  We say that $X\subseteq E$ is a \emph{set of clones} if
$a,b\in X$ are clones whenever $a\ne b$.  A \emph{cyclic} set of $M$
is a set $X$ that is a union of circuits, that is, $M|X$ has no
coloops.  A \emph{cyclic flat} is a flat that is cyclic.  It is easy
to prove that, for the set $\mathcal{Z}_M$ of cyclic flats of $M$, we
have, for $Y\subseteq E$,
\begin{equation}\label{eq:rankviacyclicflats}
  r_M(Y) = \min\{r_M(Z)+|Y-Z|\,:\,Z\in\mathcal{Z}_M\}.
\end{equation}

\begin{lemma}\label{lem:shownatural}
  Let $\rho$, $E$, $E'$, $X_i$, and $X_A$ be as above.  A matroid $M$
  on $E'$ is the natural matroid $M_\rho$ of $\rho$ if and only if
  $\mathcal{Z}_M \subseteq\{X_A\,:\,A\subseteq E\}$ and
  $r_M(X_A)=\rho(A)$ whenever $X_A\in \mathcal{Z}_M$.
\end{lemma}

\begin{proof}
  Above we showed that $r_{M_\rho}(X_A)=\rho(A)$ for all
  $A\subseteq E$.  Also, $X_i$ is a set of clones of $M_\rho$, so if
  $C$ is a circuit of $M_\rho$ and $a\in C\cap X_i$, then
  $(C-a)\cup b$, for each $b\in X_i-C$, is a circuit of $M_\rho$, and
  so $X_i\subseteq \cl_{M_\rho}(C)$.  Thus,
  $\mathcal{Z}_{M_\rho} \subseteq\{X_A\,:\,A\subseteq E\}$.

  To prove the converse, assume that
  $\mathcal{Z}_M \subseteq\{X_A\,:\,A\subseteq E\}$ and
  $r_M(X_A)=\rho(A)$ whenever $X_A\in \mathcal{Z}_M$.  By
  construction, the rank function of $M_\rho$ is given by
  \begin{align*}
    r_{M_\rho}(Y)
    = &\, \min\{\rho(A)+|Y-X_A|\,:\,A\subseteq E\}\\
    = &\, \min\{r_M(X_A)+|Y-X_A|\,:\,A\subseteq E\}
  \end{align*}
  for all $Y\subseteq E'$.  Now
  \begin{itemize}
  \item if $Y,W\subseteq E'$, then $r_M(Y) \leq r_M(W)+|Y-W|$,
  \item for each $Y$, some cyclic flat $W$ yields equality in that
    inequality, and
  \item $\mathcal{Z}_M\subseteq \{X_A\,:\,A\subseteq E\}$.
  \end{itemize}
  Thus, Equation (\ref{eq:rankviacyclicflats}) gives
  $$r_M(Y) = \min\{r_M(X_A)+|Y-X_A|\,:\,A\subseteq E\}.$$
  Thus, $M$ and $M_\rho$ have the same rank function and so are equal,
  as claimed.
\end{proof}

The containment
$\mathcal{Z}_{M_\rho} \subseteq\{X_A\,:\,A\subseteq E\}$ is proper
since each set $X_i$ is independent.

Two elements are clones in $M$ if and only if they are in exactly the
same cyclic flats of $M$, so we get the following corollary.

\begin{cor}\label{cor:shownatural}
  Let $\rho$, $E$, $E'$, $X_i$, and $X_A$ be as above.  A matroid $M$
  on $E'$ is $M_\rho$ if and only if each set $X_i$ is a set of clones
  and $r_M(X_A)=\rho(A)$ for all $X_A\in\mathcal{Z}_M$.
\end{cor}

With Corollary \ref{cor:shownatural}, it follows that the natural
matroid defined above is the same as that obtained by the construction
of iterated principal extensions followed by deletion that is given in
the proof of \cite[Theorem 11.1.9]{oxley}, and which justifies the
geometric view of the natural matroid that is mentioned after Theorem
\ref{thm:reppoly}.

It follows easily from Corollary \ref{cor:shownatural}, or from the
rank functions, that the operations of deletion and taking the natural
matroid commute: if $i\in E$, then
$M_{\rho_{\del i}} = M_\rho\del X_i$.  The same is not true of
contraction.  For $i\in E$ and each $j\in E-i$, fix a subset $Y_j$ of
any $\rho(\{i,j\})-\rho(i)$ elements of $X_j$, and let $E'_{/i}$ be
the union of all such sets $Y_j$.  It follows from Corollary
\ref{cor:shownatural} that $M_{\rho_{/i}} = M_\rho/X_i|E'_{/i}$.  From
Corollary \ref{cor:shownatural}, we also get
$M_{\rho_1\oplus\rho_2}=M_{\rho_1}\oplus M_{\rho_2}$ for integer
polymatroids $\rho_1$ and $\rho_2$; so an integer polymatroid $\rho$
on $E$ with $|E|>1$ is connected if and only if $\rho$ has no loops
and $M_\rho$ is connected.  The number of elements in the natural
matroid is the sum of all terms $\rho(i)$ for $i\in E$, so, for a
positive integer $k$, the natural matroid of the $k$-dual of an
integer $k$-polymatroid $\rho$ can have fewer, the same number of, or
more elements compared to the natural matroid of $\rho$.

Theorem \ref{thm:firstgen}, which we used to construct the natural
matroid, is the key to defining an important matroid operation,
namely, matroid union (see \cite{oxley, Welsh}).  Let
$M_1,M_2,\ldots,M_k$ be matroids on $E$.  Their \emph{matroid union},
denoted $M_1\vee M_2\vee \cdots\vee M_k$, is the matroid on $E$ having
the rank function $r'$ where, for $Y\subseteq E$,
$$r'(Y) = \min\{ r_{M_1}(X)+ r_{M_2}(X)+\cdots+ r_{M_k}(X)
+|Y-X|\,:\,X\subseteq Y\}.$$ The independent sets of
$M_1\vee M_2\vee \cdots\vee M_k$ are the sets of the form
$I_1\cup I_2\cup\cdots\cup I_k$ where $I_j$ is independent in $M_j$.
The matroids $M_1,M_2,\ldots,M_k$ also give an integer polymatroid on
$E$: the function $\rho$ on $2^E$ where, for $X\subseteq E$,
$$\rho(X) = r_{M_1}(X)+ r_{M_2}(X)+\cdots+ r_{M_k}(X),$$ 
is an integer $k$-polymatroid on $E$.  We write this as
$\rho = r_{M_1}+ r_{M_2}+\cdots+ r_{M_k}$ for brevity.  We call the
multiset $\{M_1,M_2,\ldots,M_k\}$ a \emph{decomposition} of $\rho$ and
we say that $\rho$ is \emph{decomposable}.  Not all integer
polymatroids are decomposable.  (See \cite{decomps} for more on this
topic.)  The next theorem identifies the natural matroid of a
decomposable integer polymatroid as a particular matroid union.

\begin{thm}\label{thm:decompunion}
  Let $\{M_1,M_2,\ldots,M_k\}$ be a decomposition of an integer
  polymatroid $\rho$ on $E$.  Let the sets $E'$, $X_i$, and $X_A$ be
  as above.  For each $j\in[k]$, construct $M'_j$ from $M_j$ by, for
  each $i\in E$, adding the elements of $X_i$ parallel to $i$, or as
  loops if $r_{M_j}(i)=0$, and then deleting $i$.  Then the natural
  matroid $M_\rho$ is the matroid union
  $M'_1\vee M'_2\vee \cdots\vee M'_k$.
\end{thm}

\begin{proof}
  For $X,Y\subseteq E'$, note that
  $ r_{M'_j}(Y\cap X) \leq r_{M'_j}(X)$ for each $j\in [k]$, and that
  $|Y-(Y\cap X)| = |Y-X|$.  Given how $M'_j$ is defined, if $X_A$ is
  the union of all sets $X_i$ such that $X_i\cap X\ne\emptyset$, then
  $ r_{M'_j}(X_A) = r_{M'_j}(X)$, for each $j\in [k]$; also,
  $|Y-X_A|\leq |Y-X|$.  It follows that the rank $r'(Y)$ of $Y$ in
  $M'_1\vee M'_2\vee \cdots\vee M'_k$ is given by
  $$r'(Y)=\min\{r_{M'_1}(X_A)+r_{M'_2}(X_A)+\cdots
  +r_{M'_k}(X_A)+|Y-X_A|\,:\, A\subseteq E\}.$$ Since
  $\rho = r_{M_1}+ r_{M_2}+\cdots+ r_{M_k}$, we get
  $r'(Y)=\min\{\rho(A) +|Y-X_A|\,:\, A\subseteq E\}$, that is,
  $r'(Y) = r_{M_\rho}(Y)$.  Thus, $M_\rho$ is
  $M'_1\vee M'_2\vee \cdots\vee M'_k$.
\end{proof}

An integer polymatroid and its natural matroid may have very different
connections to important classes of matroids.  For instance, for the
binary integer polymatroid on the set of seven lines of the projective
plane $PG(2,2)$ using the construction in Theorem \ref{thm:reppoly},
the natural matroid is $U_{3,14}$, which is not binary.  (The integer
$2$-polymatroids having natural matroids that are binary are
characterized in \cite{natbin}.)  In contrast, the next example and
result give links between transversal, or Boolean, polymatroids and
transversal matroids.

\begin{example}\label{bpmexample}
  If $\rho=r_{M_1}+ r_{M_2}+\cdots+r_{M_k}$ where $r(M_h)\leq 1$ for
  all $h\in[k]$, then $\rho$ is called a \emph{Boolean} polymatroid.
  Helgason \cite{helgason} introduced Boolean polymatroids, calling
  them \emph{covering hypermatroids}.  Some authors call them
  \emph{transversal polymatroids} \cite{conca, kochol, stef6}.  The
  class of Boolean polymatroids is closed under minors; Mat\'u\v{s}
  \cite{Matus} found their excluded minors.  By Theorem
  \ref{thm:decompunion} and the result that a matroid is transversal
  if and only if it is a matroid union of rank-$1$ matroids (see,
  e.g., \cite[Proposition 11.3.7]{oxley}), it follows that the natural
  matroid of a Boolean polymatroid is transversal.

  Another way to see this is via graphs.  A Boolean polymatroid
  $\rho=r_{M_1}+ r_{M_2}+\cdots+r_{M_k}$ has the following
  reformulation using a bipartite graph $G_\rho$.  Assume that
  $E\cap[k]=\emptyset$. The vertex set of $G_\rho$ is $E\cup [k]$, and
  $G_\rho$ has an edge $eh$ if and only if $r_{M_h}(e)=1$.  The rank
  $\rho(A)$ of a set $A\subseteq E$ is the cardinality of the set
  $N(A)$ of neighbors of $A$.  The natural matroid is the transversal
  matroid that is obtained from $G_\rho$ by replacing each element
  $e\in E$ by $\rho(e)$ elements, each of which is adjacent to all
  neighbors of $e$.  (See Figure \ref{fig:NatOfBoolean}.)

  Loopless Boolean $2$-polymatroids have received much attention, in
  part due to another connection with graphs.  Given such a
  polymatroid $\rho=r_{M_1}+ r_{M_2}+\cdots+r_{M_k}$, the graph $G$
  has an edge $e\in E$ incident with a vertex $h\in[k]$ if and only if
  $r_{M_h}(e)=1$.  Then $\rho(A)$, for $A\subseteq E$, is the number
  of vertices that are incident with at least one edge in $A$.  The
  natural matroid of $\rho$ is the bicircular matroid of the graph
  $G'$ that is obtained from $G$ by putting a new edge parallel to
  each nonloop edge of $G$.

  Using $G_\rho$, we see that an integer polymatroid $(E,\rho)$ is
  Boolean if and only if, for some $k$, there is a map
  $N:E\to 2^{[k]}$ with $\rho(X) =\bigl|\bigcup_{x\in X}N(x)\bigr|$
  for all $X\subseteq E$.  (This is Helgason's definition in
  \cite{helgason}.)  Given a set of subsets of $[k]$, there is an
  isomorphism from the lattice of all unions of those sets onto the
  lattice of cyclic flats of a transversal matroid so that the size of
  each union is the rank of its image.  This gives the following
  variant of Theorem \ref{thm:reppoly}.
\end{example}

\begin{figure}
  \centering
  \begin{tikzpicture} [scale=1, inner sep=0.9mm,
    Vertex/.style={circle,draw=black,fill=black!20,thick}]
    \node (a) at (0,1.4) [Vertex] {}; %
    \node (b) at (2,1.4) [Vertex] {};%
    \node (c) at (4,1.4) [Vertex] {};%
    \node (1) at (-0.25,0) [Vertex] {};%
    \node (2) at (0.5,0) [Vertex] {};%
    \node (3) at (1.25,0) [Vertex] {};%
    \node (4) at (2,0) [Vertex] {};%
    \node (5) at (2.75,0) [Vertex] {};%
    \node (6) at (3.5,0) [Vertex] {};%
    \node (7) at (4.25,0) [Vertex] {};%
    \foreach \from/\to in
    {1/a,2/a,3/a,3/b,4/a,4/b,5/b,5/c,6/b,6/c,7/c} \draw[very
    thick](\from)--(\to);%
    \node at (0,1.75) {\small $1$};%
    \node at (2,1.75) {\small $2$};%
    \node at (4,1.75) {\small $3$};%
    \node at (-0.25,-0.4) {\small $e_1$};%
    \node at (0.5,-0.4) {\small $e_2$};%
    \node at (1.25,-0.4) {\small $e_3$};%
    \node at (2,-0.4) {\small $e_4$};%
    \node at (2.75,-0.4) {\small $e_5$};%
    \node at (3.5,-0.4) {\small $e_6$};%
    \node at (4.25,-0.4) {\small $e_7$};%
    \node at (2,-1) {(a)};%
  \end{tikzpicture}
  \hspace{20pt}
  \begin{tikzpicture} [scale=1, inner sep=0.9mm,
    Vertex/.style={circle,draw=black,fill=black!20,thick}]
    \node (a) at (0,1.4) [Vertex] {}; %
    \node (b) at (2,1.4) [Vertex] {};%
    \node (c) at (4,1.4) [Vertex] {};%
    \node (1) at (-0.25,0) [Vertex] {};%
    \node (2) at (0.2,0) [Vertex] {};%
    \node (3) at (0.65,0) [Vertex] {};%
    \node (3x) at (1.1,0) [Vertex] {};%
    \node (4) at (1.55,0) [Vertex] {};%
    \node (4x) at (2,0) [Vertex] {};%
    \node (5) at (2.45,0) [Vertex] {};%
    \node (5x) at (2.9,0) [Vertex] {};%
    \node (6) at (3.35,0) [Vertex] {};%
    \node (6x) at (3.8,0) [Vertex] {};%
    \node (7) at (4.25,0) [Vertex] {};%
    \foreach \from/\to in {1/a,2/a,3/a,3/b,4/a,4/b,5/b,5/c,6/b,6/c,
      3x/a,3x/b,4x/a,4x/b,5x/b,5x/c,6x/b,6x/c,7/c} \draw[very
    thick](\from)--(\to);%
    \node at (0,1.75) {\small $1$};%
    \node at (2,1.75) {\small $2$};%
    \node at (4,1.75) {\small $3$};%
    \node at (-0.25,-0.4) {\small $x_1$};%
    \node at (0.2,-0.4) {\small $x_2$};%
    \node at (0.65,-0.4) {\small $x_3$};%
    \node at (1.1,-0.4) {\small $y_3$};%
    \node at (1.55,-0.4) {\small $x_4$};%
    \node at (2,-0.4) {\small $y_4$};%
    \node at (2.45,-0.4) {\small $x_5$};%
    \node at (2.9,-0.4) {\small $y_5$};%
    \node at (3.35,-0.4) {\small $x_6$};%
    \node at (3.8,-0.4) {\small $y_6$};%
    \node at (4.25,-0.4) {\small $x_7$};%
    \node at (2,-1) {(b)};%
  \end{tikzpicture}

  \begin{tikzpicture}[scale=1]
    \draw[thick, black!20](-2,1)--(0,0)--(2,1);%
    \draw[ultra thick, black](-2,1.05)--(-0.8,0.48);%
    \draw[ultra thick, black](-2,0.97)--(-0.4,0.2);%
    \draw[ultra thick, black](2,1.05)--(0.8,0.48);%
    \draw[ultra thick, black](2,0.97)--(0.4,0.2);%
    \filldraw (-2,1.07) node {} circle (3.3pt);%
    \filldraw (-2.03,0.93) node {} circle (3.3pt);%
    \filldraw (2,1) node {} circle (3.3pt);%
    \node at (-2,1.34) {\small $e_1$};%
    \node at (-2,0.6) {\small $e_2$};%
    \node at (-1,0.8) {\small $e_3$};%
    \node at (-1,0.25) {\small $e_4$};%
    \node at (2,1.34) {\small $e_7$};%
    \node at (1,0.8) {\small $e_5$};%
    \node at (1,0.25) {\small $e_6$};%
    \node at (-1.75,-0.25) {\small $\rho$};%
    \node at (0,-0.5) {(c)};%
  \end{tikzpicture}
  \hspace{40pt}
  \begin{tikzpicture}[scale=1]
    \draw[thick, black](-2,1)--(0,0)--(2,1);%
    \filldraw (0.4,0.2) node {} circle (2.5pt);%
    \filldraw (0.8,0.4) node {} circle (2.5pt);%
    \filldraw (1.2,0.6) node {} circle (2.5pt);%
    \filldraw (1.6,0.8) node {} circle (2.5pt);%
    \filldraw (2,1) node {} circle (2.5pt);%
    \filldraw (-0.4,0.2) node {} circle (2.5pt);%
    \filldraw (-0.8,0.4) node {} circle (2.5pt);%
    \filldraw (-1.2,0.6) node {} circle (2.5pt);%
    \filldraw (-1.6,0.8) node {} circle (2.5pt);%
    \filldraw (-2,1.07) node {} circle (2.5pt);%
    \filldraw (-2.03,0.93) node {} circle (2.5pt);%
    \node at (-2,1.34) {\small $x_1$};%
    \node at (-2,0.6) {\small $x_2$};%
    \node at (-1.6,1.05) {\small $x_3$};%
    \node at (-0.8,0.65) {\small $x_4$};%
    \node at (-1.2,0.85) {\small $y_3$};%
    \node at (-0.4,0.45) {\small $y_4$};%
    \node at (2,1.3) {\small $x_7$};%
    \node at (1.6,1.05) {\small $y_6$};%
    \node at (1.2,0.85) {\small $x_6$};%
    \node at (0.8,0.65) {\small $y_5$};%
    \node at (0.4,0.45) {\small $x_5$};%
    \node at (1.75,-0.25) {\small $M_\rho$};%
    \node at (0,-0.5) {(d)};%
  \end{tikzpicture}
  \caption{For the Boolean polymatroid $\rho$ shown in part (c), part
    (a) shows its associated bipartite graph $G_\rho$, as in Example
    \ref{bpmexample}.  Part (b) shows the bipartite graph that gives
    the transversal matroid that is the natural matroid $M_\rho$,
    which is shown in part (d).}
  \label{fig:NatOfBoolean}
\end{figure}
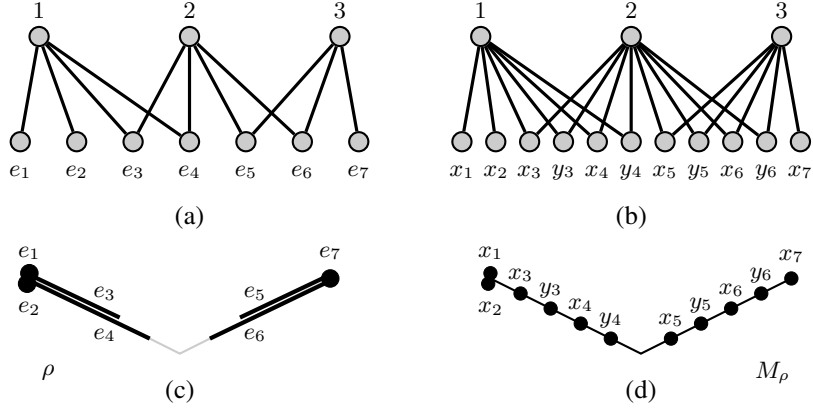

\begin{thm}
  A polymatroid $\rho$ on $E$ is Boolean if and only if there is a
  transversal matroid $M$ and map $\phi:E\to\mathcal{Z}_M$ with
  $\rho(X) = r_M(\bigcup_{i\in X}\phi(i))$ for all $X\subseteq E$.
\end{thm}

Most transversal matroids, such as $U_{2,3}$, are not Boolean
polymatroids, so the codomain of the map $\phi:E\to\mathcal{Z}_M$
cannot be extended to the lattice of flats of $M$.

\section{Bases of an integer polymatroid and its natural
  matroid}\label{sec:bases}

Independent vectors and bases of integer polymatroids are discussed,
for instance, by Herzog and Hibi in \cite{HerzogHibi}.  In this
section, where we focus solely on integer polymatroids, we show how
relating the bases of an integer polymatroid to the bases of its
natural matroid makes transparent some characterizations of integer
polymatroids that use bases.

A basis $B$ of a matroid $M$ on $E=[n]$ is a subset of $[n]$ and so
can be represented by its characteristic vector $\mathbf{b}$, the
$n$-tuple of $0$s and $1$s in which entry $i$, denoted $b_i$, is $1$
if and only if $i\in B$.  No basis contains a loop, so $b_i\leq r(i)$.
Let $\mathbf{e}_i$ be the characteristic vector of the singleton
$\{i\}$.  For the characteristic vector $\mathbf{b}$ of a basis $B$
and a basis $B'=(B-i)\cup j$ obtained by an exchange, the
characteristic vector of $B'$ is
$\mathbf{b}-\mathbf{e}_i+\mathbf{e}_j$.

The \emph{norm} of $\mathbf{v} \in \mathbb{N}^n$ is
$|\mathbf{v}|=v_1+v_2+\cdots+v_n$.  For $\mathbf{u}$ and $\mathbf{v}$
in $\mathbb{N}^n$, we write $\mathbf{u}\leq \mathbf{v}$ if
$u_i\leq v_i$ for all $i\in[n]$; also, $\mathbf{u}< \mathbf{v}$ if
$\mathbf{u}\leq \mathbf{v}$ and $\mathbf{u}\ne \mathbf{v}$.  With this
order, $\mathbb{N}^n$ is a lattice; meet and join are given by
component-wise min and max, respectively.

A definition of an integer polymatroid that is equivalent to the
definition in Section \ref{sec:intro} is that an integer polymatroid
$P$ is a nonempty finite subset $\mathbf{I}$ of $\mathbb{N}^n$, for
some $n$, for which
\begin{itemize}
\item[(I1)] if $\mathbf{v}\in \mathbf{I}$ and
  $\mathbf{u}\in\mathbb{N}^n$ with $\mathbf{u} \leq \mathbf{v}$, then
  $\mathbf{u}\in \mathbf{I}$, and
\item[(I2)] if $\mathbf{u},\mathbf{v}\in \mathbf{I}$ with
  $|\mathbf{u}|<|\mathbf{v}|$, then there is a $\mathbf{w}$ in
  $\mathbf{I}$ with
  $\mathbf{u}<\mathbf{w}\leq \mathbf{u}\vee \mathbf{v}$.
\end{itemize}
(To extend this to all polymatroids, replace $\mathbb{N}$ by
$\mathbb{R}_{\geq 0} = \{x\in\mathbb{R}\,:\,x\geq 0\}$ and require
$\mathbf{I}$ to be compact, rather than finite.)  The vectors in
$\mathbf{I}$ are the \emph{independent vectors} of $P$.  A
\emph{basis} of $P$ is a vector $\mathbf{v}\in \mathbf{I}$ for which
there is no $\mathbf{u}\in \mathbf{I}$ with $\mathbf{v}<\mathbf{u}$.
Property (I2) gives $|\mathbf{v}|=|\mathbf{u}|$ for all bases
$\mathbf{v}$ and $\mathbf{u}$ of $P$.

We now relate this notion to the definition given in Section
\ref{sec:intro}.  Let $E=[n]$.  For $\mathbf{v}\in\mathbb{N}^n$ and
$X\subseteq E$, let $|\mathbf{v}|_X$ be $\sum_{i\in X}v_i$, the sum of
the entries in $\mathbf{v}$ that are indexed by the elements in $X$.
The rank function $\rho:2^E\to \mathbb{N}$ of an integer polymatroid
$P$ on $E$ whose set of independent vectors is $\mathbf{I}$ and whose
set of bases is $\mathbf{B}$ is given by
\begin{equation}\label{eqn:Ptorank}
  \rho(X) = \max\{|\mathbf{u}|_X\,:\, \mathbf{u}\in \mathbf{I} \}
  = \max\{|\mathbf{u}|_X\,:\, \mathbf{u}\in \mathbf{B}\}
\end{equation}
for $X\subseteq E$.  The function $\rho$ satisfies properties (1)--(3)
in Section \ref{sec:intro} and so is the rank function of an integer
polymatroid.  Conversely, given an integer polymatroid
$\rho:2^E\to\mathbb{N}$, the set
\begin{equation}\label{eqn:ranktoP}
  \mathbf{I} = \{\mathbf{u}\in\mathbb{N}^n\,:\, |\mathbf{u}|_X\leq \rho(X)
  \text{ for all } X\subseteq E\}
\end{equation}
satisfies properties (I1) and (I2).  (For a proof, see \cite[p.\ 340,
Lemma 5]{Welsh}.)  Also, the maps $\mathbf{I}\mapsto \rho$ and
$\rho\mapsto \mathbf{I}$ are inverses of each other.  (See
\cite[Corollaries 44.3f and 44.3g]{sch}.)

Given an integer polymatroid $\rho$ on $E=[n]$, let $E'$, $X_i$, and
$X_A$, for $i\in E$ and $A\subseteq E$, and the natural matroid
$M_\rho$ be defined as above.  The \emph{type vector} of a subset $V$
of $E'$ is the vector $\mathbf{v}\in \mathbb{N}^n$ with
$v_i=|V\cap X_i|$ for all $i\in E$.  We use $\mathbf{T}(V)$ to denote
the type vector of $V$.  By Corollary \ref{cor:indepinnatmtd} and
Equation (\ref{eqn:ranktoP}), a subset $V$ of $E'$ is independent in
$M_\rho$ if and only if $\mathbf{T}(V)$ is an independent vector of
$\rho$, and so $V$ is a basis of $M_\rho$ if and only if
$\mathbf{T}(V)$ is a basis of $\rho$.

\begin{example}\label{ex:lppms}
  Consider a Boolean polymatroid
  $\rho=r_{M_1}+ r_{M_2}+\cdots+r_{M_k}$ on $[n]$ where each $M_h$ has
  rank $1$ and its rank-$1$ elements are consecutive integers
  $a_h,a_h+1,\ldots, b_h$, with $a_1\leq a_2\leq \cdots \leq a_k$ and
  $b_1\leq b_2\leq \cdots \leq b_k$.  The matroids
  $M_1, M_2,\ldots,M_k$ correspond to the rows in a lattice path
  diagram, where north steps are labeled by their first coordinate,
  the lower left corner is $(1,0)$, and the upper right corner is
  $(n,k)$.  (See Figure \ref{fig:LPM}.)  Bases correspond to lattice
  paths: entry $u_i$ in a basis $\mathbf{u}$ is the number of north
  steps in the corresponding path that are labeled $i$.  An elementary
  argument (as in the proof of \cite[Theorem 3.3]{lpm1}) shows that
  the correspondence between bases and lattice paths is bijective.
  Schweig \cite{Jay,jay2} introduced these \emph{lattice path
    polymatroids}.  The description of the natural matroid of a
  Boolean polymatroid in Example \ref{bpmexample} along with the ideas
  in \cite[Section 6.1]{lpm2} show that the natural matroid of a
  lattice path polymatroid is a lattice path matroid (see \cite{lpm1}
  for these matroids).  Like the class of lattice path matroids, that
  of lattice path polymatroids is closed under minors; the excluded
  minors for lattice path polymatroids are found in \cite{LPPExMin}.
  Unlike the class of lattice path matroids, that of lattice path
  polymatroids is not closed under duality.  Also, most lattice path
  matroids are not lattice path polymatroids.
\end{example}

\begin{figure}
  \centering
  \begin{tikzpicture}[scale=0.7]
    \draw (1,0) grid (4,1);%
    \draw (3,1) grid (6,2);%
    \draw (5,2) grid (7,3);%
    \draw[line width = 2 pt] (1,0) -- (4,0) -- (4,2) -- (5,2) --
    (6,2)--(6,3) -- (7,3);%
    \draw (0.75,0.5) node {$1$};%
    \draw (1.75,0.5) node {$2$};%
    \draw (2.75,0.5) node {$3$};%
    \draw (3.75,0.5) node {$4$};%
    \draw (2.75,1.5) node {$3$};%
    \draw (3.75,1.5) node {$4$};%
    \draw (4.75,1.5) node {$5$};%
    \draw (5.75,1.5) node {$6$};%
    \draw (4.75,2.5) node {$5$};%
    \draw (5.75,2.5) node {$6$};%
    \draw (6.75,2.5) node {$7$};%
  \end{tikzpicture}
  \caption{A lattice path diagram for the Boolean polymatroid in
    Figure \ref{fig:NatOfBoolean}.  In the notation of Example
    \ref{ex:lppms}, we have $a_1=1$, $a_2=3$, $a_3=5$, $b_1=4$,
    $b_2=6$, and $b_3=7$.  The highlighted path corresponds to the
    basis $(0,0,0,2,0,1,0)$.}
  \label{fig:LPM}
\end{figure}
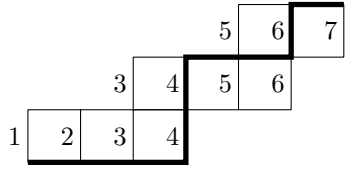

The following characterizations of integer polymatroids by bases are
known (see, e.g., \cite{HerzogHibi}).  We provide a transparent way to
see this and similar results using the natural matroid.

\begin{thm}\label{thm:CharBases}
  A nonempty set $\mathbf{B}\subseteq \mathbb{N}^n$ is the set of
  bases of an integer polymatroid on $E=[n]$ if and only if either of
  the following equivalent conditions holds:
  \begin{itemize}
  \item[(B)] if $\mathbf{u},\mathbf{v}\in \mathbf{B}$ with $u_i>v_i$
    for some $i\in [n]$, then there is a $j\in [n]$ for which
    $u_j<v_j$ and $\mathbf{u}-\mathbf{e}_i +\mathbf{e}_j$ is in
    $\mathbf{B}$,
  \item[(B$'$)] if $\mathbf{u},\mathbf{v}\in \mathbf{B}$ with $u_i>v_i$
    for some $i\in [n]$, then there is a $j\in [n]$ for which
    $u_j<v_j$ and both $\mathbf{u}-\mathbf{e}_i +\mathbf{e}_j$ and
    $\mathbf{v}-\mathbf{e}_j +\mathbf{e}_i$ are in $\mathbf{B}$.
  \end{itemize}
\end{thm}

\begin{proof}
  First let $\mathbf{B}$ be the set of bases of an integer polymatroid
  $\rho$.  We prove property (B$'$), which implies property (B).  Fix
  $\mathbf{u},\mathbf{v}\in \mathbf{B}$ with $u_i>v_i$ for some
  $i\in [n]$.  Let $U$ and $V$ be bases of the natural matroid
  $M_\rho$ of $\rho$ with $\mathbf{T}(U)=\mathbf{u}$ and
  $\mathbf{T}(V)=\mathbf{v}$.  Since each set $X_t$ is a set of
  clones, we may assume that $V\cap X_t\subseteq U\cap X_t$ whenever
  $v_t\leq u_t$.  Fix $x\in (U-V) \cap X_i$.  By the symmetric basis
  exchange property for matroids, there is an element $y\in V-U$, say
  in $X_j$, so that both $(U-x)\cup y$ and $(V -y)\cup x$ are bases of
  $M_\rho$, so $\mathbf{u}-\mathbf{e}_i+\mathbf{e}_j$ and
  $\mathbf{v} -\mathbf{e}_j+\mathbf{e}_i$ are in $\mathbf{B}$.  Now
  $V\cap X_j\not\subseteq U\cap X_j$, so, as needed, $u_j<v_j$.

  To finish the proof, we show that a nonempty subset $\mathbf{B}$ of
  $\mathbb{N}^n$ that satisfies property (B) is the set of bases of an
  integer polymatroid.  For each $i\in E$, set
  $m_i=\max\{u_i\,:\,\mathbf{u} \in\mathbf{B}\}$ and let $X_i$ be a
  set of $m_i$ elements where $X_i\cap X_j=\emptyset$ if $i\ne j$.
  Set $X_A=\cup_{i\in A}X_i$, for $A\subseteq E$, and $E'=X_E$.  Let
  $\mathcal{B}=\{B\,:\,B\subseteq E' \text{ and } \mathbf{T}(B) \in
  \mathbf{B}\}$. Take $U,V \in \mathcal{B}$ with $x\in U-V$; say
  $x\in X_i$.  If there is an element $y$ in $(V-U)\cap X_i$, then
  $(U-x)\cup y$ has the same type vector as $U$ and so is in
  $\mathcal{B}$.  Now assume that $V\cap X_i\subsetneq U\cap X_i$.
  Let $\mathbf{u}=\mathbf{T}(U)$ and $\mathbf{v}=\mathbf{T}(V)$.
  Thus, $u_i>v_i$.  By property (B), there is a $j\in [n]$ for which
  $u_j<v_j$ and $\mathbf{w}=\mathbf{u}-\mathbf{e}_i +\mathbf{e}_j$ is
  in $\mathbf{B}$.  Thus, $(V-U)\cap X_j\ne \emptyset$.  For any
  $y\in (V-U)\cap X_j$, the set $W=(U-x)\cup y$ has type vector
  $\mathbf{w}$, so $W\in \mathcal{B}$.  Thus, $\mathcal{B}$ is the set
  of bases of a matroid $M$ on $E'$.  Define $\rho:2^E\to \mathbb{N}$
  by $\rho(A)=r_M(X_A)$.  Thus, $\rho$ is an integer polymatroid on
  $E$.  Also, $X_i$ is a set of clones in $M$.  It now follows from
  the definition of $\rho$ and Corollary \ref{cor:shownatural} that
  $M$ is the natural matroid of $\rho$.  From the definition of $M$
  and the comments before Example \ref{ex:lppms}, we have that
  $\mathbf{B}$ is the set of bases of $\rho$, as needed.
\end{proof}

The strategy we used above adapts to prove integer-polymatroid
counterparts of other axiom schemes for matroids that use bases or
independent sets.  We cite just one example, for the middle basis
property.

\begin{thm}
  A nonempty set $\mathbf{B}\subseteq \mathbb{N}^n$ is the set of
  bases of an integer polymatroid on $E=[n]$ if and only if the
  following two conditions hold:
  \begin{itemize}
  \item if $\mathbf{u},\mathbf{v}\in \mathbf{B}$ with
    $\mathbf{u}\ne \mathbf{v}$, then $\mathbf{u}\not\leq \mathbf{v}$
    and $\mathbf{v}\not\leq \mathbf{u}$, and
  \item whenever $\mathbf{x},\mathbf{y}\in \mathbb{N}^n$ with
    $\mathbf{x}\leq \mathbf{y}$ and there are
    $\mathbf{u},\mathbf{v}\in \mathbf{B}$ with
    $\mathbf{x}\leq \mathbf{u}$ and $\mathbf{v}\leq \mathbf{y}$, then
    there is some $\mathbf{w}\in \mathbf{B}$ with
    $\mathbf{x}\leq \mathbf{w}\leq \mathbf{y}$.
  \end{itemize}
\end{thm}

For a positive integer $k$, certain properties of the $k$-dual
$\rho^*$ of an integer $k$-polymatroid $\rho$ highlight how natural
$k$-duality is.  For instance, generalizing a result of Kung
\cite{JPSK} for matroids, Whittle \cite{GeoffDual} showed that the map
$\rho\mapsto \rho^*$ is the only involution on the class of integer
$k$-polymatroids that switches deletion and contraction, i.e.,
$(\rho_{\del i})^*=(\rho^*)_{/i}$ and
$(\rho_{/i})^*=(\rho^*)_{\del i}$ for all $i\in E$.  The set of bases
of the dual of a matroid $M$ on $E$ is given by
$\{E-B\,:\,B\in \mathcal{B}\}$ where $\mathcal{B}$ is the set of bases
of $M$; the next result generalizes this to the $k$-dual of an integer
$k$-polymatroid.

\begin{thm}\label{thm:dualbases}
  Let $k$ be a positive integer and let $\rho$ be an integer
  $k$-polymatroid on $E=[n]$, with $\mathbf{B}$ its set of bases.  The
  set $\mathbf{B}^*$ of bases of the $k$-dual $\rho^*$ is
  $\{\mathbf{u}^*\,:\,\mathbf{u}\in\mathbf{B}\}$ where
  $\mathbf{u}^*=(k,k,\ldots,k)-\mathbf{u}$.
\end{thm}

\begin{proof}
  Note that $|\mathbf{u}|=\rho(E)$ if and only if
  $|\mathbf{u}^*|=k|E|-\rho(E)=\rho^*(E)$.  With this, the equivalence
  of the following statements shows that $\mathbf{u} \in\mathbf{B}$ if
  and only if $\mathbf{u}^* \in\mathbf{B}^*$:
  \begin{itemize}
  \item $|\mathbf{u}|_A\leq\rho(A)$ for all $A\subseteq E$,
  \item $|\mathbf{u}|_{E-A}\leq \rho(E-A)$ for all $A\subseteq E$,
  \item $\rho(E)-|\mathbf{u}|_A\leq \rho(E-A)$ for all $A\subseteq E$,
  \item $k|A|-|\mathbf{u}|_A\leq k|A|-\rho(E)+\rho(E-A)$ for all
    $A\subseteq E$,
  \item $|\mathbf{u}^*|_A\leq \rho^*(A)$ for all $A\subseteq E$.\qedhere
  \end{itemize}
\end{proof}

\section{Circuits of an integer polymatroid and its natural
  matroid}\label{sec:circuits}

We next develop a theory of circuits for integer polymatroids that is
analogous to that for bases in Section \ref{sec:bases}.  Just as the
bases of an integer polymatroid $\rho$ are the type vectors of the
bases of the natural matroid $M_\rho$, so the circuits of $\rho$ are
the type vectors of the circuits of $M_\rho$, with one exception:
loops of $\rho$ map to the empty set in $M_\rho$.  This is addressed
by the ambient set $\mathbf{U}$ that we consider below.  Another issue
that we must address so that the circuits determine the integer
polymatroid is that we need the rank of each element $i$ for which
$X_i$ is a set of coloops of $M_\rho$; the set $\mathbf{U}$ also takes
care of this.  (In matroids, such elements have rank one, but in
integer polymatroids, the rank could be any positive integer.) Recall
that $[n]_0$ denotes the set $\{0,1,\ldots,n\}$.

The \emph{circuits} of an integer polymatroid $\rho$ on $E=[n]$ are
the vectors $\mathbf{u}$ in the set
$$\mathbf{U}=[\rho(1)]_0\times [\rho(2)]_0\times\cdots \times
[\rho(n)]_0$$ that are not independent and each vector $\mathbf{w}$
with $\mathbf{w}<\mathbf{u}$ is independent.  Thus, from the set
$\mathbf{C}$ of circuits of $\rho$, the set of independent vectors of
$\rho$ is
\begin{equation}\label{eqn:indepfromcir}
  \mathbf{I} = \{\mathbf{u}\in\mathbf{U}\,:\,
  \text{ there is no } \mathbf{c}\in\mathbf{C} \text{ with }
  \mathbf{c}\leq \mathbf{u}\}.
\end{equation}
By the remarks before Example \ref{ex:lppms}, a vector
$\mathbf{u} \in \mathbf{U}$ is a circuit of $\rho$ if and only if some
(equivalently, every) set $C$ in the natural matroid $M_\rho$ with
$\mathbf{u}=\mathbf{T}(C)$ is a circuit of $M_\rho$.  See Figure
\ref{fig:BadCirCompls} for an example.  The set $\mathbf{C}$ is an
antichain in $\mathbb{N}^n$.  Recall that all antichains in
$\mathbb{N}^n$ are finite.

\begin{figure}
  \centering
\begin{tikzpicture}[scale=1.5]
  \draw[very thick, black](0,0)--(2,0);
  \draw[very thick, black](0,0.4)--(2,0.4);

  \filldraw (1,0) node {} circle  (1.8pt);

  \node at (0.2,-0.2) {$3$};
  \node at (1,-0.25) {$2$};
  \node at (1,0.6) {$1$};
\end{tikzpicture}
\caption{The circuits of this rank-$3$ integer polymatroid are
  $(2,0,2)$, $(2,1,1)$, and $(0,1,2)$.}
  \label{fig:BadCirCompls}
\end{figure}
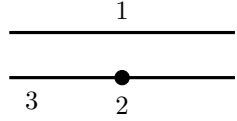

While $\mathbf{U}$ gives the rank of each element, the next lemma
shows how, from $\mathbf{C}$ alone, to get the rank of any element $i$
for which there is a $\mathbf{c}\in\mathbf{C}$ with $c_i>0$.

\begin{lemma}\label{lem:CircuitsCaptureRank}
  Let $\rho$ be an integer polymatroid on $E=[n]$.  For $i\in E$ with
  $\rho(i)>0$, the set $X_i$ is a subset of a circuit of the natural
  matroid $M_\rho$ if and only if $\rho(E)<\rho(E-i)+\rho(i)$.  In
  this case, $\rho(i) = \max\{c_i\,:\, \mathbf{c}\in \mathbf{C}\}$.
\end{lemma}

\begin{proof}
  First assume that $\rho(E)<\rho(E-i)+\rho(i)$.  Fix $a\in X_i$.
  Since $X_i$ is independent in $M_\rho$, so is $X_i-a$.  Extend
  $X_i-a$ to a basis $B$ of $M_\rho\del a$, which, by the assumed
  inequality, is a basis of $M_\rho$.  Let $C$ be the fundamental
  circuit of $a$ with respect to the basis $B$ of $M_\rho$.  Then
  $X_i\subseteq C$, for if $b\in X_i-C$, the subset $(C-a)\cup b$ of
  the basis $B$ would be a circuit of $M_\rho$ since $a$ and $b$ are
  clones, but that is a contradiction.

  To prove the contrapositive of the converse, assume that
  $\rho(E)=\rho(E-i)+\rho(i)$.  Thus,
  $r_{M_\rho}(E')=r_{M_\rho}(E'-X_i)+r_{M_\rho}(X_i)$, from which we
  get $M_\rho=(M_{\rho}\del X_i)\oplus(M_{\rho}|X_i)$.  With this
  direct sum decomposition and the fact that $X_i$ is independent in
  $M_\rho$, it follows that $X_i$ is disjoint from all circuits of
  $M_\rho$.
\end{proof}

The following result will be used in the next section.

\begin{lemma}\label{cyclicgivescircuit}
  Let $\rho$ be an integer polymatroid on $E=[n]$.  For $A\subseteq E$
  and $i\in A$, we have $\rho(A)<\rho(A-i)+\rho(i)$ if and only if
  there is a circuit $\mathbf{u}\in\mathbf{C}$ with $u_i>0$ and
  $u_j=0$ for all $j\in E-A$.
\end{lemma}

\begin{proof}
  Given $\mathbf{u}\in\mathbf{C}$ with $u_i>0$ and $u_j=0$ for all
  $j\in E-A$, for each $a\in X_i$ there is a circuit $C$ of $M_\rho$
  with $\mathbf{u}=\mathbf{T}(C)$ and $a\in C$.  Thus, $a$ is not a
  coloop of $M|X_A$, so $\rho(A)<\rho(A-i)+\rho(i)$.  The converse
  follows by applying Lemma \ref{lem:CircuitsCaptureRank} to
  $\rho_{\del E-A}$.
\end{proof}

Lemma \ref{lem:CircuitsCaptureRank} and the remarks before it lead to
the following setting for characterizations of integer polymatroids
via circuits, as in Theorem \ref{thm:circchar}: the set of circuits is
a subset of a set of the form
$\mathbf{U}=[m_1]_0\times [m_2]_0\times\cdots \times [m_n]_0$ where
each $m_i$ is a nonnegative integer.  Recording $m_i$ is the only way
we get $\rho(i)$ when all elements of $X_i$ (if there are any) are
coloops of $M_\rho$.  Two circuits in the natural matroid may have
different type vectors or the same type vector; therefore there are
two circuit elimination properties, (C3) and (C4).

\begin{thm}\label{thm:circchar}
  Let $m_1,m_2,\ldots,m_n$ be nonnegative integers and let
  $\mathbf{C}$ be a subset of
  $[m_1]_0\times [m_2]_0\times \cdots \times [m_n]_0$ where, for each
  $i\in [n]$, if $u_i>0$ for some $\mathbf{u}\in\mathbf{C}$, then
  $m_i= \max\{u_i\,:\, \mathbf{u}\in \mathbf{C}\}$.  The set
  $\mathbf{C}$ is the set of circuits of an integer polymatroid on
  $E=[n]$ if and only if $\mathbf{C}$ satisfies properties
  \emph{(C1)--(C4):}
  \begin{itemize}
  \item[(C1)] each vector in $\mathbf{C}$ has at least two positive
    entries,
  \item[(C2)] if $\mathbf{u},\mathbf{v}\in \mathbf{C}$ with
    $\mathbf{u}\ne \mathbf{v}$, then $\mathbf{u}\not< \mathbf{v}$ and
    $\mathbf{v}\not< \mathbf{u}$,
  \item[(C3)] if $\mathbf{u},\mathbf{v}\in \mathbf{C}$ with
    $\mathbf{u}\ne \mathbf{v}$ and if $u_i,v_i>0$, then there is a
    $\mathbf{z}\in\mathbf{C}$ so that
    $\mathbf{z}<\mathbf{u}\vee \mathbf{v}$ and
    $z_i< \max(u_i,v_i)$, and
  \item[(C4)] if $\mathbf{u}\in \mathbf{C}$, and $i\in E$ with
    $0<u_i<m_i$, and $j\in E-i$ with $0<u_j$, then there is a
    $\mathbf{v}\in \mathbf{C}$ with $v_i=u_i+1$, with $v_h\leq u_h$
    for all $h\ne i$, and with $v_j<u_j$.
  \end{itemize}
  Thus, an integer polymatroid on $[n]$ is a pair
  $(\mathbf{U},\mathbf{C})$ where
  \begin{itemize}
  \item[(i)]
    $\mathbf{U} = [m_1]_0\times [m_2]_0\times \cdots \times [m_n]_0$,
    for some $m_1,m_2,\ldots,m_n$ in $\mathbb{N}$,
  \item[(ii)] $\mathbf{C} \subseteq \mathbf{U}$ and $\mathbf{C}$
    satisfies properties \emph{(C1)--(C4)}, and
  \item[(iii)] if $i\in E$ and $u_i>0$ for some
    $\mathbf{u}\in \mathbf{C}$, then
    $m_i= \max\{u_i\,:\, \mathbf{u}\in \mathbf{C}\}$.
  \end{itemize}
  If $\rho$ is the rank function of the integer polymatroid given by
  the pair $(\mathbf{U},\mathbf{C})$, then $\rho(i)=m_i$ for all
  $i\in E$.
\end{thm}

\begin{proof}
  Let $\mathbf{C}$ be the set of circuits of an integer polymatroid
  $\rho$ on $E$.  By Lemma \ref{lem:CircuitsCaptureRank}, if
  $\mathbf{u}\in\mathbf{C}$ and $u_i>0$, then $m_i=\rho(i)$.  Property
  (C1) holds since each set $X_i$ is independent in $M_\rho$.  If
  $\mathbf{u},\mathbf{v}\in \mathbf{C}$ and $\mathbf{u}< \mathbf{v}$,
  then a circuit $V$ of $M_\rho$ with $\mathbf{v}=\mathbf{T}(V)$ and a
  subset $U$ of $V$ with $\mathbf{u}=\mathbf{T}(U)$ would be
  comparable circuits of $M_\rho$; this contradiction proves property
  (C2).
  
  For property (C3), take $\mathbf{u},\mathbf{v}\in \mathbf{C}$ with
  $\mathbf{u}\ne \mathbf{v}$ and $u_i,v_i>0$.  Let $U$ and $V$ be
  circuits of $M_\rho$ with $\mathbf{u}=\mathbf{T}(U)$ and
  $\mathbf{v}=\mathbf{T}(V)$.  Each set $X_j$ is a set of clones in
  $M_\rho$, so we may assume that if $u_j\leq v_j$, then
  $U\cap X_j\subseteq V\cap X_j$, and if $v_j\leq u_j$, then
  $V\cap X_j\subseteq U\cap X_j$.  Thus,
  $U\cap V\cap X_i\ne \emptyset$.  For any $a\in U\cap V\cap X_i$,
  circuit elimination applied to $U$ and $V$ gives a circuit $C$ of
  $M_\rho$ with $C\subseteq (U\cup V)-a$.  The inclusions that we have
  assumed give $|C\cap X_j|\leq \max(u_j,v_j)$ for all $j\in E$, and
  $|C\cap X_i|<\max(u_i,v_i)$, as needed.
  
  For property (C4), consider $\mathbf{u} \in \mathbf{C}$ and $i\in E$
  with $0< u_i<\rho(i)$.  Let $C$ be a circuit of $M_\rho$ with
  $\mathbf{T}(C)=\mathbf{u}$.  Fix $a\in C\cap X_i$ and $b\in X_i-C$,
  and set $C'=(C-a)\cup b$, which is a circuit of $M_\rho$ since $a$
  and $b$ are clones.  For $j\in E-i$ with $u_j>0$, fix
  $c\in C\cap X_j$.  By circuit elimination, $M_\rho$ has a circuit
  $D$ with $D\subseteq (C\cup C')-c= (C\cup b)-c$.  Property (C2)
  applied to $\mathbf{u}$ and $\mathbf{T}(D)$ forces
  $(C\cup b)\cap X_i\subseteq D$.  Thus, $|D\cap X_i|=u_i+1$ and
  $|D\cap X_h|\leq u_h$ for all $h\ne i$, and the inequality is strict
  for $h=j$, so property (C4) holds.
  
  For the converse, assume that the pair $(\mathbf{U},\mathbf{C})$
  satisfies properties (i)--(iii).  For each $i\in E$, let $X_i$ be a
  set of size $m_i$ with $X_i\cap X_j=\emptyset$ whenever $i\ne j$.
  We use $X_A$ and $E'$ as above.  Let
  $\mathcal{C}=\{ C\,:\, C\subseteq E' \text{ and
  }\mathbf{T}(C)\in\mathbf{C}\}$.  Now $\emptyset\not\in\mathcal{C}$
  by property (C1).  For two sets $C$ and $C'$ in $\mathcal{C}$,
  either (i) $\mathbf{T}(C)=\mathbf{T}(C')$, so for at least one
  $i\in E$, the subsets $C\cap X_i$ and $C'\cap X_i$ of $X_i$ are
  different but have the same size, or (ii)
  $\mathbf{T}(C)\ne\mathbf{T}(C')$, so by property (C2), there are
  $i,j\in E$ with $|C\cap X_i|<|C'\cap X_i|$ and
  $|C'\cap X_j|<|C\cap X_j|$; thus, neither $C$ nor $C'$ contains the
  other.

  We next show that $\mathcal{C}$ satisfies the circuit elimination
  property.  Take two sets $U,V \in \mathcal{C}$ and $a\in U\cap V$;
  say $a\in X_j$.  Let $\mathbf{u}=\mathbf{T}(U)$ and
  $\mathbf{v}=\mathbf{T}(V)$.  If $\mathbf{u}\ne\mathbf{v}$, then by
  property (C3), there is a $\mathbf{z}\in \mathbf{C}$ with
  $\mathbf{z}<\mathbf{u}\vee\mathbf{v}$ and $z_j< \max(u_j,v_j)$.
  Clearly $(U\cup V)-a$ has a subset $C$ with
  $\mathbf{T}(C)=\mathbf{z}$, as needed.  Now assume that
  $\mathbf{u}=\mathbf{v}$.  If $U\cap X_j\ne V\cap X_j$, then there is
  an element $b\in (V-U)\cap X_j$, and the set $C=(U-a)\cup b$ has
  $\mathbf{T}(C)=\mathbf{u}$, so $C\in \mathcal{C}$, as needed.  If
  $U\cap X_j= V\cap X_j$, then $U\cap X_i\ne V\cap X_i$ for some
  $i\in E-j$, and so $0<u_i<m_i$.  By property (C4), since $u_j>0$,
  there is an $\mathbf{x}\in \mathbf{C}$ with $x_i=u_i+1$, with
  $x_h\leq u_h$ for all $h\ne i$, and with $x_j<u_j$.  Clearly
  $(U\cup V)-a$ has a subset $C$ with $\mathbf{T}(C)=\mathbf{x}$, as
  needed.

  Thus, $\mathcal{C}$ is the set of circuits of a matroid $M$ on $E'$.
  As in the proof of Theorem \ref{thm:CharBases}, from $M$, we get an
  integer polymatroid $\rho$ whose natural matroid is $M$.  Since
  $\mathbf{C}$ is the set of circuits of $\rho$, this completes the
  proof.
\end{proof}

The set $\mathbf{C} = \{(4,1),(2,2)\}$ satisfies all properties except
(C4), so property (C4) does not follow from properties (C1)--(C3).

Let $k$ be a positive integer.  Let $\mathbf{H}$ be the set of the
type vectors of the hyperplanes of the natural matroid of an integer
$k$-polymatroid.  Let
$\mathbf{C}^* =\{(k,k,\ldots,k)-\mathbf{u}\,:\,\mathbf{u}\in
\mathbf{H}\}$.  In contrast to Theorem \ref{thm:dualbases},
$\mathbf{C}^*$ might not be the set of circuits of an integer
polymatroid.  For instance, for the integer $2$-polymatroid $\rho$ in
Figure \ref{fig:BadCirCompls}, we have
$$\mathbf{C}^* = \{(2,1,0),(0,2,2),(1,1,2),(1,2,1)\},$$ and properties
(C3) and (C4) fail.

For an integer polymatroid $\rho$ on $E=[n]$ and any $i\in E$, since
$\rho_{/i}(j) = \rho(\{i,j\})-\rho(i)$ for all $j\in E-i$, the
circuits of the contraction $\rho_{/i}$ are contained in the Cartesian
product
$$\mathbf{U}_{/i}=\prod_{j\in E-i}[\rho(\{i,j\})-\rho(i)]_0.$$
Let $\mathbf{C}'$ be the set of circuits of $\rho$ with the $i$th
entry deleted from each vector.  Since the circuits of a contraction
$M/Y$ of a matroid $M$ are the minimal nonempty sets of the form $C-Y$
as $C$ ranges over the circuits of $M$, and the natural matroid
$M_{\rho_{/i}}$ of $\rho_{/i}$ is $M_\rho/X_i|E'_{/i}$ (in the
notation used after Corollary \ref{cor:shownatural}), it follows that
the circuits of $\rho_{/i}$ are the minimal vectors in
$\mathbf{C}'\cap \mathbf{U}_{/i}$ that have at least two positive entries.

We noted in Section \ref{sec:nat} that an integer polymatroid $\rho$
on $E$ with $|E|>1$ is connected if and only if $\rho$ has no loops
and $M_\rho$ is connected.  Thus, an integer polymatroid $\rho$ on
$[n]$ is connected if and only if for each pair of distinct integers
$i,j\in[n]$, there is a circuit $\mathbf{u}$ of $\rho$ with $u_i>0$
and $u_j>0$.

\section{Flats, cyclic sets, and cyclic flats in
  polymatroids}\label{sec:cyclic}

While some results in this section apply only to integer polymatroids,
many apply to all polymatroids.  To describe what we do in this
section, we first need some definitions.  Flats in a polymatroid
$\rho$ on $E$ are defined as in matroids: a subset $A$ of $E$ is a
\emph{flat} of $\rho$ if $ \rho(A \cup i) > \rho(A)$ for all
$i \in E-A$.  Let $\mathcal{F}_\rho$ denote the set of flats of
$\rho$.  Unless we are focusing only on matroids, $\mathcal{F}_\rho$
does not determine $\rho$ since, for instance, $\rho$ and $c\, \rho$,
for any positive real $c$, have the same flats.

There are various equivalent ways to say that a set $X$ in a matroid
$M$ is cyclic, including:
\begin{itemize}
\item[(i)] $X$ is a union of circuits;
\item[(ii)] $M|X$ has no coloops;
\item[(iii)] $r(X) < r(X - y) + r(y)$ for each $y \in X$ that is not a
  loop.
\end{itemize}
As in \cite{cyclicpoly}, we adapt condition (iii) to define cyclic
sets in a polymatroid $\rho$ on $E$: a subset $A$ of $E$ is
\emph{cyclic} if $\rho(A) < \rho(A - i) + \rho(i)$ for all $i \in A$
with $\rho(i) > 0$.  We let $\mathcal{Y}_\rho $ denote the set of all
cyclic sets of $\rho$.

Of greatest interest are the cyclic flats, that is, the flats that are
cyclic.  The set of cyclic flats of $\rho$ is denoted
$\mathcal{Z}_\rho$, or $\mathcal{Z}_M$ for a matroid $M$.  As in the
case of matroids, $\mathcal{Z}_\rho$ is a lattice under
inclusion. (See the comment after Lemma \ref{lem:cycwclprops}.)  The
next result, from \cite{juliethesis, cyc}, characterizes matroids in
terms of their cyclic flats and the ranks of those sets.

\begin{thm}\label{thm:cfaxiom}
  For a pair $(\mathcal{Z}, r)$, where $\mathcal{Z}\subseteq 2^E$ and
  $r : \mathcal{Z} \to \mathbb{N}$, there is a matroid $M$ for which
  $\mathcal{Z} = \mathcal{Z}_M$ and $r(Z)=r_M(Z)$ for all
  $Z\in\mathcal{Z}$ if and only if
  \begin{itemize}
  \item[(Z0)] ordered by inclusion, $\mathcal{Z}$ is a lattice,
  \item[(Z1)] $r(\hat{0}_{\mathcal{Z}}) = 0$, where
    $\hat{0}_{\mathcal{Z}}$ is the least element of $\mathcal{Z}$,
  \item[(Z2)] $0 < r(B) - r(A) < |B - A|$ for all sets $A, B$ in
    $\mathcal{Z}$ with $A \subsetneq B$, and
  \item[(Z3)]
    $r(A \vee B) + r(A \wedge B) + |(A \cap B) - (A \wedge B)| \leq
    r(A) + r(B)$ for all $A, B$ in $\mathcal{Z}$.
  \end{itemize}
\end{thm}

Csirmaz \cite{cyclicpoly} extended this theorem.  His result, stated
next, characterizes polymatroids using cyclic flats and the value of
the rank function on each of those flats as well as on each singleton
set.  The rank of each element must be given since, while in a matroid
each element that is not in the least cyclic flat (the set of loops)
has rank $1$, in a polymatroid, such an element may have any positive
rank.

\begin{thm}\label{thm:pcfaxiom}
  For a pair $(\mathcal{Z}, \rho')$, where $\mathcal{Z}\subseteq 2^E$
  and $\rho': \mathcal{Z} \cup E \to \mathbb{R}_{\geq 0}$, there is a
  polymatroid $\rho$ on $E$ with $\mathcal{Z} = \mathcal{Z}_\rho$ and
  $\rho(x)=\rho'(x)$ for all $x\in\mathcal{Z} \cup E$ if and only if
  \begin{enumerate}
  \item[(PZ0)] ordered by inclusion, $\mathcal{Z}$ is a lattice,
  \item[(PZ1)] the least element of $\mathcal{Z}$, denoted
    $\hat{0}_{\mathcal{Z}}$, is $\{i \in E : \rho'(i) = 0\}$, and
    $\rho'(\hat{0}_{\mathcal{Z}}) = 0$,
  \item[(PZ2)] for all sets $A, B$ in $\mathcal{Z}$ with
    $A\subsetneq B$, 
    $$0 < \rho'(B) - \rho'(A) < \sum_{i\in B-A} \rho'(i),$$ 
  \item[(PZ3)] for all sets $A, B$ in $\mathcal{Z}$, 
    $$\rho'(A \vee B) + \rho'(A \wedge  B) + \sum_{i\in (A\cap B)-(A\wedge B)}
    \rho'(i) \leq \rho'(A) + \rho'(B),$$
    and
  \item[(PZ4)] if $A \in \mathcal{Z}$ and $i\in A$, then
    $\rho'(i)\leq \rho'(A)$.
  \end{enumerate}
\end{thm}

We will show how, in the case of an integer polymatroid $\rho$,
Theorem \ref{thm:pcfaxiom} follows from Theorem \ref{thm:cfaxiom}; we
do this by relating the flats of $\rho$ to those of its natural
matroid, and likewise for cyclic sets and for cyclic flats.  The proof
of Theorem \ref{thm:pcfaxiom} in \cite{cyclicpoly} has the same
general outline as the proof of Theorem \ref{thm:cfaxiom} that Sims
\cite{juliethesis} gave.  In particular, for the more involved
implication, assuming that the properties above hold for
$(\mathcal{Z}, \rho')$, one defines a function
$\rho:2^E\to\mathbb{R}$, checks that the defining properties of a
polymatroid hold, and shows that its cyclic flats are precisely the
sets in $\mathcal{Z}$, and that $\rho$ and $\rho'$ have the same values
on the sets in $\mathcal{Z}$ and the elements of $E$.  The function
$\rho$ is defined by
$$\rho(A) = \min\{\rho'(X) + \sum_{i\in A-X}\rho'(i) \,:\, X \in
\mathcal{Z}\}.$$ This makes it natural to consider, for a polymatroid
$\rho$ on $E$ and subset $A$ of $E$, the set
\begin{equation}\label{eq:calR}
  \mathcal{R}_\rho(A) = \{B\in\mathcal{Z}_\rho\,:\,\rho(A)=
  \rho(B)+\sum_{i\in A-B}\rho(i)\}.
\end{equation}
For a matroid $M$, we write this set as $\mathcal{R}_M(A)$.  Our main
new result is Theorem \ref{thm:Rstructure}, where we show that
$\mathcal{R}_\rho(A)$ is a sublattice of $\mathcal{Z}_\rho$, we
identify its least and greatest elements, and we show that each pair
of elements in $\mathcal{R}_\rho(A)$ is a modular pair.  To prepare
for that, we develop basic results about flats and cyclic sets, and
two operators related to them.  (While some of these results may be
known, we include proofs for completeness.)

We start with flats.  The flats of an integer polymatroid are related
to those of its natural matroid in the simplest possible way, as the
next lemma states.

\begin{lemma}\label{lem:flatnat}
  For an integer polymatroid $\rho$ on $E$, let $\hat{0}_\rho$ be
  $\{i\in E\,:\,\rho(i)=0\}$.  A subset $A$ of $E$ is a flat of $\rho$
  if and only if $\hat{0}_\rho\subseteq A$ and $X_A$ is a flat of the
  natural matroid $M_\rho$.
\end{lemma}

\begin{proof}
  Assume that $\hat{0}_\rho\subseteq A$ and that $X_A$ is a flat of
  $M_\rho$.  If $i\in E-A$, then there are elements $b\in X_i$, and
  $r_{M_{\rho}}(X_{A\cup i})\geq r_{M_{\rho}}(X_A\cup b) >
  r_{M_{\rho}}(X_A)$, so $\rho(A\cup i)>\rho(A)$, as needed.  We now
  prove the contrapositive of the converse.  If
  $\hat{0}_\rho\not\subseteq A$, then clearly $A$ is not a flat of $\rho$.
  Assume that $X_A$ is not a flat of $M_\rho$, so
  $r_{M_{\rho}}(X_A\cup b) = r_{M_{\rho}}(X_A)$ for some
  $b \in E'-X_A$; say $b\in X_i$.  Then
  $r_{M_{\rho}}(X_A\cup c) = r_{M_{\rho}}(X_A)$ for all $c\in X_i$
  since $X_i$ is a set of clones.  From this, repeatedly applying
  submodularity gives $r_{M_{\rho}}(X_{A\cup i}) = r_{M_{\rho}}(X_A)$,
  so $\rho(A\cup i) = \rho(A)$, so $A$ is not a flat of $\rho$.
\end{proof}

\begin{cor}\label{cor:Flattice}
  The set $\mathcal{F}_\rho$ of flats of an integer polymatroid
  $\rho$, ordered by inclusion, is isomorphic to a sublattice of the
  lattice $\mathcal{F}_{M_\rho}$.  The meet of two flats of $\rho$ is
  their intersection.
\end{cor}

By \cite[Theorem 2.1]{cyc}, every finite lattice is isomorphic to the
lattice of cyclic flats of a matroid.  With that and the construction
in Theorem \ref{thm:reppoly}, it follows that, in contrast to
matroids, every finite lattice is isomorphic to the lattice of flats
of an integer polymatroid.

\begin{lemma}\label{lem:rholattice}
  For a polymatroid $\rho$ on $E$, the intersection of two flats is a
  flat, so, ordered by inclusion, $\mathcal{F}_\rho$ is a lattice.
\end{lemma}

\begin{proof}
  Fix $A,B\in \mathcal{F}_\rho$ and $e\in E-(A\cap B)$; say
  $e\not\in A$.  From submodularity and these assumptions,
  $\rho\bigl((A\cap B)\cup e\bigr) - \rho(A\cap B) \geq \rho(A\cup
  e)-\rho(A)>0$, as needed.
\end{proof}

This lemma justifies extending the definition of the closure operator
from matroids to polymatroids.  The \emph{closure operator
  $\cl_\rho : 2^E \to 2^E$ of a polymatroid} $\rho$ on $E$ is given by
\begin{equation}\label{eqn:dfnclosure}
  \cl_\rho(A) = \bigcap\{F \,:\, F \in \mathcal{F}_\rho \text{ and } A
  \subseteq F\}
\end{equation}
for $A \subseteq E$; equivalently, $\cl_\rho(A)$ is
the minimum flat (with respect to inclusion) that is a superset of
$A$.  Several results follow immediately:
$\cl_\rho(A) \in \mathcal{F}_\rho$ by Lemma \ref{lem:rholattice}, the
image of $\cl_\rho$ is $\mathcal{F}_\rho$, and $\cl_\rho$ is a closure
operator in the general sense, that is, (i) $A \subseteq \cl_\rho(A)$
for all $A \subseteq E$, (ii) if $A \subseteq B \subseteq E$, then
$\cl_\rho(A) \subseteq \cl_\rho(B)$, and (iii)
$\cl_\rho(\cl_\rho(A)) = \cl_\rho(A)$ for all $A \subseteq E$.  The
MacLane-Steinitz exchange property of matroid closure operators fails
for most polymatroids; for instance, in the integer polymatroid $\rho$
in Figure \ref{fig:NatOfBoolean}, (c), we have
$e_2\in \cl_\rho(e_3)-\cl_\rho(\emptyset)$ but
$e_3\not\in\cl_\rho(e_2)$.

\begin{lemma}\label{lem:clreform}
  Let $\rho$ be a polymatroid on $E$.  If $A\subseteq E$, then
  $\cl_\rho(A) = \{i \,:\, \rho(A \cup i) = \rho(A)\}$ and
  $\rho(A) = \rho(\cl_\rho (A))$.
\end{lemma}

\begin{proof}
  Let $X=\{i \,:\, \rho(A \cup i) = \rho(A)\}$.  Now $A\subseteq X$.
  Repeated use of submodularity gives $\rho(A) = \rho(X)$.  If
  $i\in E-X$, then $\rho(X\cup i)\geq \rho(A\cup i)>\rho(A)=\rho(X)$,
  so $X$ is a flat.  Let $F$ be a flat with $A\subseteq F$.  Now
  $X\subseteq F$ since, for any $i\in X$, from
  $ \rho(A \cup i) = \rho(A)$ we get $\rho(F \cup i) = \rho(F)$ by
  submodularity.  Thus, $\cl_\rho(A)= X$.
\end{proof}

We next give properties of the closure operator that are special to
integer polymatroids.

\begin{lemma}\label{lem:clprops}
  For an integer polymatroid $\rho$ on $E$, let $\hat{0}_\rho$ be
  $\{i\,:\,\rho(i)=0\}$.  For $A\subseteq E$,
  \begin{enumerate}
  \item $\cl_\rho(A)=B$ if and only if $\hat{0}_\rho\subseteq B$ and
    $\cl_{M_\rho}(X_A)=X_B$, and
  \item $\cl_\rho(A)=A\cup \hat{0}_\rho \cup C_A$ where $C_A$ is the
    set of all $i\in E$ for which some circuit $\mathbf{u}$ of $\rho$
    has $u_i=1$ and $u_j=0$ for all $j\in E-(A\cup i)$.
  \end{enumerate}
\end{lemma}

\begin{proof}
  Each set $X_i$ is a set of clones of $M_\rho$, so
  $\cl_{M_\rho}(X_A)$ is the smallest flat $X_B$ that contains $X_A$.
  Part (1) follows from this observation and Lemma \ref{lem:flatnat}.
  For part (2), clearly $A\cup \hat{0}_\rho \subseteq \cl_\rho(A)$.
  Fix $i\in C_A$ and a circuit $\mathbf{u}$ with $u_i=1$ and $u_j=0$
  for all $j\in E-(A\cup i)$.  The circuits $C$ of $M_\rho$ with
  $\mathbf{T}(C)= \mathbf{u}$ show that
  $X_i\subseteq \cl_{M_\rho}(X_A)$; thus, $i\in \cl_\rho(A)$, and so
  $A\cup \hat{0}_\rho\cup C_A\subseteq \cl_\rho(A)$.  For the other
  inclusion, fix a basis $D$ of $M_\rho|X_A$, so $D$ is also a basis
  of $M_\rho|\cl_{M_\rho}(X_A)$.  If
  $i\in \cl_\rho(A)-(A\cup \hat{0}_\rho)$ and $a\in X_i$, then the
  type vector of the fundamental circuit of $a$ with respect to $D$
  shows that $i\in C_A$.
\end{proof}

We now turn to cyclic sets.  We first focus on integer polymatroids.

\begin{lemma}\label{lem:cycnat}
  Let $\rho$ be an integer polymatroid on $E$.  For $A\subseteq E$,
  statements \emph{(1)--(3)} are equivalent:
  \begin{itemize}
  \item[(1)] $A$ is a cyclic set of $\rho$,
  \item[(2)] $X_A$ is a cyclic set of $M_\rho$,
  \item[(3)] for each $i\in A$, either $\rho(i)=0$ or there is a
    circuit $\mathbf{u}$ of $\rho$ for which $u_i>0$ and $u_j=0$ for
    all $j\in E-A$.
  \end{itemize}
\end{lemma}

\begin{proof}
  Assume that statement (1) holds.  For any $a\in X_A$, there is an
  $i\in A$ with $\rho(i)>0$ and $a\in X_i$.  If $a$ were a coloop of
  $M_\rho|X_A$, then all elements of $X_i$ would be coloops of
  $M_\rho|X_A$, contrary to having $\rho(A) < \rho(A - i) + \rho(i)$.
  Thus, statement (2) holds.
  
  Assume that statement (2) holds.  Fix $i\in A$ with $\rho(i)>0$.  No
  $a\in X_i$ is a coloop of $M_\rho|X_A$, so some circuit $C$ of
  $M_\rho$ has $a\in C\subseteq X_A$.  Statement (3) now follows.

  By Lemma \ref{cyclicgivescircuit}, statement (3) implies statement
  (1).
\end{proof}

We can expand the list of equivalent conditions for $X$ being a cyclic
set of a matroid $M$ (items (i)--(iii) in the second paragraph of this
section):
\begin{itemize}
\item[(iv)] $X$ is a union of cocircuits of the dual $M^*$,
\item[(v)] $E-X$ is an intersection of hyperplanes of $M^*$, and
\item[(vi)] $E-X$ is a flat of $M^*$.
\end{itemize}
The flats of $M^*$, ordered by inclusion, form a geometric lattice, so
the cyclic sets of $M$, ordered by inclusion, form a lattice, the
order-dual of which is geometric.  Thus, we have the following
corollary of Lemma \ref{lem:cycnat}.

\begin{cor}\label{cor:Ylattice}
  For an integer polymatroid $\rho$ on $E$, its set $\mathcal{Y}_\rho$
  of cyclic sets, ordered by inclusion, is a lattice.  The join of two
  cyclic sets is their union.
\end{cor}

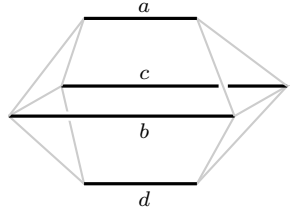
\begin{figure}
  \centering
  \begin{tikzpicture}[scale=1]
    \draw[thick,black!20](0,1)--(1,0.1)--(2.5,0.1)--(3,1);%
    \draw[thick,black!20](0.7,1.4)--(1,0.1)--(2.5,0.1)--(3.7,1.4);%
    \draw[very thick](1,0.1)--(2.5,0.1);%
    \draw[line width = 3.5pt,white](0.6,1)--(1.2,1);%
    \draw[thick,black!20](0,1)--(3,1)--(3.7,1.4)--(0.7,1.4)--(0,1);%
    \draw[very thick](0,1)--(3,1);%
    \draw[very thick](3.7,1.4)--(0.7,1.4);%
    \draw[line width = 3.5pt,white](2.85,1.3)--(2.85,1.5);%
    \draw[thick,black!20](0,1)--(1,2.3)--(2.5,2.3)--(3,1);%
    \draw[thick,black!20](0.7,1.4)--(1,2.3)--(2.5,2.3)--(3.7,1.4);%
    \draw[very thick](1,2.3)--(2.5,2.3);%

    \node at (1.8,2.45) {\footnotesize $a$};%
    \node at (1.8,0.8) {\footnotesize $b$};%
    \node at (1.8,1.55) {\footnotesize $c$};%
    \node at (1.8,-0.1) {\footnotesize $d$};%
  \end{tikzpicture}
  \caption{A $2$-polymatroid counterpart of the V\'amos matroid.  Each
    pair of lines is coplanar except $a$, $d$. }
  \label{fig:vamos}
\end{figure}

If $\rho$ is not a matroid, then the order dual of $\mathcal{Y}_\rho$
need not be a geometric lattice, as one can see from the
$2$-polymatroid counterpart of the V\'amos matroid shown in Figure
\ref{fig:vamos}, where the only sets not in $\mathcal{Y}_\rho$ are the
singleton sets and $\{a,d\}$.

Lemma 1 of \cite{cyclicpoly} shows that every flat of a polymatroid
contains a maximum cyclic flat.  By the next result and the discussion
below it, a similar statement holds for all sets, and it comes from a
property that generalizes Corollary \ref{cor:Ylattice}.

\begin{lemma}\label{lem:cycprop}
  Let $\rho$ be a polymatroid on $E$.
  \begin{itemize}
  \item[(1)] If $X,Y\in\mathcal{Y}_\rho$, then
    $X\cup Y\in\mathcal{Y}_\rho$.  Thus, under inclusion,
    $\mathcal{Y}_\rho$ is a lattice.
  \item[(2)] If $X\in\mathcal{Y}_\rho$, then
    $\cl_\rho(X)\in\mathcal{Y}_\rho$, and so
    $\cl_\rho(X)\in\mathcal{Z}_\rho$.
  \end{itemize}
\end{lemma}

\begin{proof}
  To prove part (1), fix $X,Y\in\mathcal{Y}_\rho$ and $i\in X\cup Y$
  with $\rho(i)>0$; say $i\in X$.  By the assumptions and
  submodularity,
  $\rho(X\cup Y) -\rho\bigl((X\cup Y)-i\bigr)\leq \rho(X) -
  \rho(X-i)<\rho(i)$, so $X\cup Y\in\mathcal{Y}_\rho$.  For part (2),
  take $X \in \mathcal{Y}_\rho$ and $i\in \cl_\rho(X)$ with
  $\rho(i)>0$.  Then, as needed,
  $\rho\bigl(\cl_\rho(X)\bigr)- \rho\bigl(\cl_\rho(X)-i\bigr)<\rho(i)$
  since the left side is $0$ if $i\not\in X$ (since
  $\rho(X) = \rho(X\cup i)$), and at most $ \rho(X)- \rho(X-i)$ if
  $i\in X$ (by submodularity).
\end{proof}

This lemma justifies making the following definition.  For a
polymatroid $\rho$ on $E$, its \emph{cyclic operator}
$\cy_\rho : 2^E \to 2^E$ is given by, for $A\subseteq E$,
$$\cy_\rho (A)
=\bigcup \{D \,:\, D \in \mathcal{Y}_\rho \text{ and } D \subseteq
A\}.$$ Thus, $\cy_\rho (A)$ is the maximum cyclic subset of $A$.  If
$A\in\mathcal{F}_\rho$, then $\cy_\rho(A)\in\mathcal{F}_\rho$ since
$\cy_\rho(A)\subseteq \cl_\rho(\cy_\rho(A)) \subseteq A$ and
$\cl_\rho(\cy_\rho(A))$ is cyclic (by part (2) of Lemma
\ref{lem:cycprop}) and so must be $\cy_\rho(A)$.  For a matroid $M$,
the cyclic set $\cy_M(A)$ is the union of the circuits that are
subsets of $A$.  The operator $\cy_M$ plays roles in recent papers,
such as \cite{cycbinary}.  Note that the image of $\cy_\rho$ is
precisely $\mathcal{Y}_\rho$.  Also, (i) if $A \subseteq E$, then
$\cy_\rho (A) \subseteq A$, (ii) if $A \subseteq B \subseteq E$, then
$\cy_\rho (A) \subseteq \cy_\rho (B)$, and (iii) if $A \subseteq E$,
then $\cy_\rho (\cy_\rho (A)) = \cy_\rho (A)$.

\begin{lemma}\label{lem:cyciprops}
  Let $\rho$ be an integer polymatroid on $E$.  For any set
  $A\subseteq E$, the set $\cy_\rho(A)$ is the union of all subsets of
  $A$ of either of the following forms:
  \begin{itemize}
  \item[(i)] $\{i\}$ with $\rho(i)=0$, or
  \item[(ii)] $S(\mathbf{u})=\{i\,:\,u_i\ne 0\}$ where $\mathbf{u}$ is
    a circuit of $\rho$ and $u_j=0$ for all $j\in E-A$.
  \end{itemize}
  Also, $\cy_\rho(A)=B$ if and only if $B$ is the maximum subset of
  $A$ with $\cy_{M_\rho}(X_A)=X_B$.
\end{lemma}

\begin{proof}
  The first assertion follows from Lemma \ref{lem:cycnat} and the
  definition of $\cy_\rho$.  That and the connection between the
  circuits of $\rho$ and those of $M_\rho$ give the second assertion.
\end{proof}

We state the next lemma, which is basic and well known, so that we
can cite it.

\begin{lemma}\label{lem:superbasic}
  Let $\rho$ be a polymatroid on $E$.  Assume that $A\subseteq E$,
  that $i\in A$, and that $\rho(A) = \rho(A - i) + \rho(i)$.  If
  $Y\subseteq A$ and $i\in Y$, then $\rho(Y) = \rho(Y - i) + \rho(i)$.
\end{lemma}

\begin{proof}
  By submodularity,
  $\rho(i)\geq \rho(Y)-\rho(Y-i)\geq \rho(A)-\rho(A-i)= \rho(i)$.
\end{proof}

The next lemma identifies the elements in $A-\cy_\rho(A)$ as the
counterparts of coloops in the deletion $\rho_{\del E-A}$.  

\begin{lemma}\label{lem:cycprops}
  Let $\rho$ be a polymatroid on $E$.  For any set $A\subseteq E$,
  \begin{enumerate}
  \item
    $\cy_\rho (A) = A - \{i \in A\, :\, \rho(i) > 0 \text{ and }
    \rho(A) = \rho(A - i) + \rho(i)\}$, and
  \item
    $\displaystyle{ \rho(A) = \rho(\cy_\rho (A)) + \sum_{i\in
        A-\cy_\rho(A)} \rho(i)}$.
  \end{enumerate}
\end{lemma}

\begin{proof}
  Let
  $X = \{i \in A\, :\, \rho(i) > 0 \text{ and } \rho(A) = \rho(A - i)
  + \rho(i)\}$.  By Lemma \ref{lem:superbasic}, no cyclic subset of
  $A$ contains any $i\in X$, so $\cy_\rho(A)\subseteq A-X$.  Part (1)
  will follow by showing that $A-X$ is cyclic.  First note that
  repeatedly applying Lemma \ref{lem:superbasic}, adding one element
  at a time to go from $A-X$ to $A$, gives
  \begin{equation}\label{eq:cyceq}
    \rho(A) = \rho(A-X) + \sum_{i\in X} \rho(i).
  \end{equation}
  If there were a $j\in A-X$ with $\rho(j)>0$ and
  $ \rho(A-X) = \rho((A-X)-j) +\rho(j)$, then this equality, Equation
  (\ref{eq:cyceq}), and submodularity would give
  $$\rho(A)-\rho(j) = \rho((A-X)-j) + \sum_{i\in X} \rho(i)\geq
  \rho(A-j).$$ This inequality is contrary to having $j\not\in X$, so
  $A-X$ is cyclic.  Part (2) follows from part (1) and Equation
  (\ref{eq:cyceq}).
\end{proof}

The next lemma is like part (2) of Lemma \ref{lem:cycprop}, but
switches flats and cyclic sets.
  
\begin{lemma}\label{lem:cycwclprops}
  For a polymatroid $\rho$ on $E$, if $A \in \mathcal{F}_\rho$, then
  $\cy_\rho (A) \in \mathcal{F}_\rho$, so
  $\cy_\rho (A) \in \mathcal{Z}_\rho$.
\end{lemma}

\begin{proof}
  Fix $A \in \mathcal{F}_\rho$ and $i\not\in\cy_\rho(A)$.  We must
  show that $\rho(\cy_\rho(A) \cup i) > \rho(\cy_\rho (A))$.  This
  holds by Lemmas \ref{lem:cycprops} and \ref{lem:superbasic} if
  $i \in A - \cy_\rho (A)$.  If $i \not\in A$, then the assumption
  $A \in \mathcal{F}_\rho$ and submodularity give
  $\rho(\cy_\rho(A) \cup i) - \rho(\cy_\rho (A))\geq \rho(A \cup i) -
  \rho(A)>0$.
\end{proof}

With Lemmas \ref{lem:cycprop} and \ref{lem:cycwclprops}, we see that
$\mathcal{Z}_\rho$ is a lattice: for $A,B\in \mathcal{Z}_\rho$, their
meet is $A\wedge B=\cy_\rho(A\cap B)$ and their join is
$A\vee B = \cl_\rho(A\cup B)$.

The next lemma, along with Lemma \ref{lem:cycprops}, is a basic tool
for investigating the sets $\mathcal{R}_\rho(A)$, which we defined in
Equation (\ref{eq:calR}).

\begin{lemma}\label{lem:recastR}
  Let $\rho$ be a polymatroid on $E$.  For any subsets $A$ and $B$ of
  $E$, the equality
  \begin{equation}\label{eq:recastR}
    \rho(A)=\rho(B)+\sum_{i\in A-B}\rho(i)
  \end{equation}
  holds if and only if
  \begin{enumerate}
  \item $\rho(A) = \rho(A-i)+\rho(i)$ for all $i\in A-B$, and
  \item $\rho(A\cap B)=\rho(B)$ (equivalently,
    $\cl_\rho(A\cap B)=\cl_\rho(B)$).
  \end{enumerate}
\end{lemma}

\begin{proof}
  First assume that properties (1) and (2) hold.  Applying Lemma
  \ref{lem:superbasic} to add one element at a time going from
  $A\cap B$ to $A$ gives
  $$\rho(A)=\rho(A\cap B)+\sum_{i\in A-B}\rho(i)$$ and replacing
  $\rho(A\cap B)$ by $\rho(B)$, as (2) justifies, yields Equation
  (\ref{eq:recastR}).

  Now assume that Equation (\ref{eq:recastR}) holds.  Repeated uses of
  submodularity give
  $$\rho(A)\leq \rho(A\cap B)+\sum_{i\in A-B}\rho(i).$$
  Also, $\rho(A\cap B)\leq \rho(B)$.  These inequalities and Equation
  (\ref{eq:recastR}) give $\rho(A\cap B)=\rho(B)$, so property (2)
  holds.  With this, for any $i\in A-B$, we have
  $$\rho(A)=\rho(A\cap B)+\Bigl(\sum_{j\in A-B, j\ne
    i}\rho(j)\Bigr)+\rho(i)\geq \rho(A-i)+\rho(i)\geq \rho(A),$$ from
  which we get $\rho(A) = \rho(A-i)+\rho(i)$, so property (1) holds.
\end{proof}

We now consider the operators $\cl$ and $\cy$ together.  Note that if
$B$ is a basis of a matroid $M$ that has neither loops nor coloops,
then $\cl(\cy(B))=\emptyset$ but $\cy(\cl(B))=E(M)$; thus, $\cl$ and
$\cy$ need not commute.  Lemmas \ref{lem:flatnat} and \ref{lem:cycnat}
give the following result.

\begin{cor}\label{cor:relatecfs}
  For an integer polymatroid $\rho$ on $E$, let $\hat{0}_\rho$ be
  $\{i\in E\,:\,\rho(i)=0\}$. For $A\subseteq E$, we have
  $A\in \mathcal{Z}_\rho$ if and only if $\hat{0}_\rho\subseteq A$ and
  $X_A\in \mathcal{Z}_{M_\rho}$.
\end{cor}

For an integer polymatroid $\rho$, since all cyclic flats of $M_\rho$
have the form $X_A$ for some $A\subseteq E$ and the map
$\phi:\mathcal{Z}_\rho\to \mathcal{Z}_{M_\rho}$ where $\phi(A)=X_A$ is
a bijection, properties that can be described via cyclic flats lift
from matroids to integer polymatroids.  With these ideas, the case of
Theorem \ref{thm:pcfaxiom} for integer polymatroids follows from
Theorem \ref{thm:cfaxiom}.

Not all properties of cyclic flats for matroids extend to
polymatroids.  For instance, for matroids, the cyclic flats of the
dual $M^*$ are the set complements of the cyclic flats of $M$, so
$\mathcal{Z}_{M^*}$ is isomorphic to the order dual of
$\mathcal{Z}_M$.  The same is not true for $k$-polymatroids and their
$k$-duals, as one can check using the example in Figure
\ref{fig:BadCirCompls} or \ref{fig:vamos}.

To conclude, we use Lemmas \ref{lem:superbasic}, \ref{lem:cycprops},
and \ref{lem:recastR} to show that $\mathcal{R}_\rho(A)$ is a
sublattice of $\mathcal{Z}_\rho$ (so the meet and join operations are
the same as in $\mathcal{Z}_\rho$), identify the least and greatest
elements of $\mathcal{R}_\rho(A)$, and show that each pair $(B,B')$ of
cyclic flats in $\mathcal{R}_\rho(A)$ is a modular pair of flats, that
is, $\rho(B)+\rho(B') =\rho(B\cup B')+\rho(B\cap B')$.  (That equality
can fail if only one of $B$ or $B'$ is in $\mathcal{R}_\rho(A)$.)

\begin{thm}\label{thm:Rstructure}
  Let $\rho$ be a polymatroid on $E$.  For any subset $A$ of $E$,
  \begin{enumerate}
  \item[(I)] $\cl_\rho(\cy_\rho(A))$ and $\cy_\rho(\cl_\rho(A))$ are
    in $\mathcal{R}_\rho(A)$,
  \item[(II)] if $B\in \mathcal{R}_\rho(A)$, then
    $\cl_\rho(\cy_\rho(A))\subseteq B \subseteq
    \cy_\rho(\cl_\rho(A))$,
  \item[(III)] $\mathcal{R}_\rho(A)$ is a sublattice of
    $\mathcal{Z}_\rho$, and
  \item[(IV)] if $B,B'\in \mathcal{R}_\rho(A)$, then $(B,B')$ is a
    modular pair of flats.
  \end{enumerate}
\end{thm}

\begin{proof} 
  When $B$ is $\cy_\rho(A)$, property (1) in Lemma \ref{lem:recastR}
  holds by Lemma \ref{lem:cycprops}, as does property (2) since
  $B\subseteq A$.  Those properties then follow when $B$ is
  $\cl_\rho\bigl(\cy_\rho(A)\bigr)$ since
  $\bigl(\cl_\rho\bigl(\cy_\rho(A)\bigr)\bigr)\cap A =\cy_\rho(A)$, so
  $\cl_\rho\bigl(\cy_\rho(A)\bigr)\in \mathcal{R}_\rho(A)$.  Those
  properties clearly also hold when $B$ is $\cl_\rho(A)$.  From this,
  when $B$ is $\cy_\rho\bigl(\cl_\rho(A)\bigr)$, we get property (1)
  by Lemma \ref{lem:superbasic}, and property (2) by applying Lemma
  \ref{lem:superbasic} as elements of $A-B$ are removed from $A$ and
  $\cl_\rho(A)$. Thus,
  $\cy_\rho\bigl(\cl_\rho(A)\bigr)\in \mathcal{R}_\rho(A)$, so part
  (I) holds.
   
  Assume that $B\in \mathcal{R}_\rho(A)$.  Property (1) of Lemma
  \ref{lem:recastR} gives $\cy_\rho(A)\subseteq B$, so, since $B$ is a
  flat, $\cl_\rho(\cy_\rho(A))\subseteq B$.  Property (2) of Lemma
  \ref{lem:recastR} and the fact that $B$ is a flat give
  $B=\cl_\rho(A\cap B)\subseteq \cl_\rho(A)$, so, since $B$ is cyclic,
  $B\subseteq \cy_\rho(\cl_\rho(A))$.  Thus, part (II) holds.
  
  For assertion (III), we start with an inequality that we will use
  below.  Let $A$ be any subset of $E$ and let $B$ and $B'$ be in
  $\mathcal{Z}_\rho$.  We claim that
  $$\sum_{i\in(B\cap B')-(B\wedge B')}\rho(i)+
  \sum_{i\in A-B}\rho(i)+\sum_{i\in A-B'}\rho(i)\geq \sum_{i\in
    A-(B\vee B')}\rho(i)+ \sum_{i\in A-(B\wedge B')}\rho(i).$$ This
  inequality holds since
  \begin{itemize}
  \item $A-(B\vee B')$ is a subset of each of $A-(B\wedge B')$, $A-B$,
    and $A-B'$ (so terms $\rho(i)$ coming from its elements appear
    twice on each side of the inequality), and
  \item
    $\bigl(A-(B\wedge B')\bigr)-\bigl(A-(B\vee B')\bigr)\subseteq
    \bigl( (B\cap B')-(B\wedge B') \bigr)\cup (A-B)\cup (A-B')$ (so
    terms $\rho(i)$ that appear once on the right side also appear on
    the left side).
  \end{itemize}
  Now assume that $B,B'\in\mathcal{R}_\rho(A)$, so
  $$\rho(B)+ \sum_{i\in A-B}\rho(i) = \rho(A) = \rho(B')+ \sum_{i\in
    A-B'}\rho(i).$$ Then, using submodularity as formulated in
  property (PZ3) of Theorem \ref{thm:pcfaxiom}, along with the
  inequality above, we have
  \begin{align*}
    2\,\rho(A)
    & = \rho(B)+\rho(B')+ \sum_{i\in A-B}\rho(i) + \sum_{i\in
      A-B'}\rho(i)\\
    & \geq  \rho(B\vee B')+\rho(B\wedge B')+
      \sum_{i\in(B\cap B')-(B\wedge B')}\rho(i)
      +\sum_{i\in A-B}\rho(i) + \sum_{i\in A-B'}\rho(i)\\
    & \geq \rho(B\vee B')+\rho(B\wedge B')+ \sum_{i\in A-(B\vee
      B')}\rho(i) + \sum_{i\in  A-(B\wedge B')}\rho(i).
  \end{align*}
  Since
  $$\rho(B\vee B')+ \sum_{i\in A-(B\vee
    B')}\rho(i) \geq \rho(A) \qquad \text{and} \qquad \rho(B\wedge
  B')+ \sum_{i\in A-(B\wedge B')}\rho(i)\geq \rho(A),$$ the inequality
  above forces these inequalities to be equalities, which proves
  assertion (III).  Moreover, all inequalities in the argument above
  must be equalities, so equality holds in (PZ3) for $B$ and $B'$.
  Now $ \rho(B\cup B') =\rho(\cl_\rho(B\cup B')) = \rho(B\vee B')$ and
  $$\rho(B\cap B') =\rho(B\wedge B')+ \sum_{i\in(B\cap B')-(B\wedge
    B')}\rho(i) $$ by part (I) since $B\cap B'\in\mathcal{F}_\rho$ and
  $B\wedge B'=\cy_\rho(B\cap B')\in \mathcal{R}_\rho(B\cap B')$, so
  assertion (IV) follows.
\end{proof}

While $\mathcal{R}_\rho(A)$ is a sublattice of $\mathcal{Z}_\rho$, it
might not be an interval in $\mathcal{Z}_\rho$, as taking $A$ to be a
basis of the Fano plane shows.  The corollary below is immediate from
property (II).

\begin{cor}\label{cor:RforCyc}
  If $A\in\mathcal{Z}_\rho$, then $\mathcal{R}_\rho(A)=\{A\}$.
\end{cor}

\bibliographystyle{alpha}

\end{document}